\documentclass{article}

\usepackage{command_MU}

\usepackage{command_alexis}
\usepackage[font=small,labelfont=bf]{caption}
\usepackage{multirow}
\usepackage{booktabs}
\usepackage{hyperref}
\usepackage{enumitem}
\usepackage{authblk}

\title{Shortest-support Multi-Spline Bases for Generalized Sampling\thanks{This work was supported in part by the European Research Council (H2020-ERC Project GlobalBioIm) under Grant 692726  and in part by the Swiss National Science Foundation, Grant 200020\_184646/1.}}

\author[1]{Alexis Goujon \thanks{alexis.goujon@epfl.ch}}
\author[1]{Shayan Aziznejad}
\author[3]{Alireza Naderi}
\author[1]{Michael Unser}

\affil[1]{\'Ecole polytechnique f\'ed\'erale de Lausanne}
\affil[3]{University of British Columbia}

\begin{document}
\maketitle
\begin{abstract}
Generalized sampling consists in the recovery of a function $f$, from the samples of the responses of a collection of linear shift-invariant systems to the input $f$. The reconstructed function is typically a member of a finitely generated integer-shift-invariant space that can reproduce polynomials up to a given degree $M$. While this property allows for an approximation power of order $(M+1)$, it comes with a tradeoff on the length of the support of the basis functions. Specifically, we prove that the sum of the length of the support of the generators is at least $(M+1)$. Following this result, we introduce the notion of shortest basis of degree $M$, which is motivated by our desire to minimize computational costs. We then demonstrate that any basis of shortest support generates a Riesz basis. Finally, we introduce a recursive algorithm to construct the shortest-support basis for any multi-spline space. It provides a generalization of both polynomial and Hermite B-splines. This framework paves the way for novel applications such as fast derivative sampling with arbitrarily high approximation power.
\end{abstract}

\section{Introduction}

\subsection{Generalized Sampling in Shift-Invariant Spaces}

Since the formulation of Nyquist-Shannon's celebrated sampling theorem \cite{Shannon1949}, the reconstruction of a function  from discrete measurements has been extended in many ways \cite{Jerri1977,Unser2000}. In particular, Papoulis proposed the framework of generalized sampling \cite{Papoulis1977}, where he showed that any bandlimited function $f$ is uniquely
determined by the sequences of discrete measurements (generalized samples)
\begin{equation}
\label{generalizedsampling}
g_n(kT) = (h_n*f)(kT) = \dotprod{f}{\psi_n(\cdot-kT)}, \quad n=1,...,N, \quad k\inZ,
\end{equation}
where $(g_n(t))_{n=1,...,N}$ are the outcome of $N$ linearly independent systems applied to $f$. The sampling is assumed to proceed at $1/N$ the Nyquist rate ({\it i.e.}, $T = NT_{\mathrm{Nyq}} = 2N\pi/\omega_{\max}$, where $\omega_{\max}$ is the maximum frequency of $f$). The functions $\psi_n(t) = h_n(-t)$, $t\inR$, are called the analysis functions. They are the time-reversed versions of the impulse responses. The sampling theorem was also generalized to many different function spaces such as integer-shift-invariant spaces \cite{Aldroubi1994,Unser1994}, including spline spaces \cite{Hummel1983,Unser1992,Aldroubi1992}. Following this extension and Papoulis' theory, Unser and Zerubia introduced a framework to perform generalized sampling without the bandlimited constraint \cite{Unser1998,Unser1997a} which includes important cases such as interlaced and derivative sampling in spline spaces. In this paper, we adopt the same framework and propose to reconstruct a function $f$ from discrete samples $g_n(k), k=1,...,N$ in an integer-shift-invariant space generated by a finite collection of generators as in some recent works \cite{Garcia2008,Pohl2012,Radha2019}. The structure of such reconstruction spaces has been thoroughly studied \cite{DeBoor1994a,Aldroubi1996,Grochenig2018} and there exist theoretical results that lead to the critical choice of relevant generating functions \cite{DeBoor1994}. As a minimal requirement to get a good approximation space, the generating functions should satisfy jointly the partition-of-unity condition \cite{DeBoor1985}. In addition, there exists a tradeoff between the approximation power of the space and the size of the support of the generating functions \cite{Blu2001}.

\subsection{Polynomial Splines}
A polynomial spline is a piecewise polynomial function defined over the real line. Of special interest are the splines of degree $n$ because they provide one free parameter per segment. They are defined by distinct knots and polynomial pieces of degree $n$ that are connected smoothly so that the global function has continuous derivatives up to order $(n-1)$. The splines whose knots are uniformly spaced are called cardinal splines and they are relevant to many applications such as image processing \cite{Unser1999}. In the 50s, Isaac Schoenberg laid the foundation of cardinal splines \cite{Schoenberg1973,schoenberg1967spline} when he showed that the set $S_n$ of cardinal splines of degree $n$ could be generated by a single function \cite{DeBoor1976}, the B-spline of degree $n$. In this paper we will consider the causal B-spline and denote it by $\beta_+
^n$. This simple building block is also the shortest nonzero spline of degree $n$. Interestingly, the B-splines can be constructed recursively with the relation
\begin{equation}
\beta_+^{n+1} = \beta_+^n*\beta_+^0,
\end{equation}
starting from $\beta_+^0$, which is the rectangular window over $[0,1)$
\begin{equation}
\beta_+^0(x) = \begin{cases}1 ,& 0\leq x<1\\
0, & \text{otherwise.}
\end{cases}
\end{equation}

The convolution by $\beta_+^0$ can be decomposed in two successive operations: an integration (which transforms a spline of degree $n$ into a spline of degree $(n+1)$) followed by a finite difference (which gives back a compactly supported function). Indeed, $(f*\beta_+^0)(x) = \Delta\{\int_{-\infty}^x f(t)\dint t\}$, where $\Delta\{f\}=(f(\cdot)-f(\cdot-1))$ is the finite difference of $f$. Along with their great reproducing properties and shortest support, B-splines allow an efficient and practical implementation, which is exploited in many fields \cite{DeBoor1972,DeBoor1980,Unser1993,Unser1993a}.

\subsection{Multi-Splines}
To perform generalized sampling, it is natural to look at multi-spline spaces since they offer additional degrees of freedom. A cardinal multi-spline space is defined as the sum of $N\inN$ spline spaces: $S_{\mathbf{n}} = S_{n_1}+\cdots+S_{n_N}$, $\mathbf{n} = ({n_1,...,n_{N}})$ and $n_1<\cdots<n_N\inN$. From now on, any spline will be assumed to be a cardinal spline unless stated otherwise. It is worth noting that, in the case of consecutive spaces specified by $n_k = n_1+(k-1)$, the resulting space is exactly the space of piecewise polynomials of degree $n_N$ that are in $C^{n_{1}-1}(\mathbb{R})$, the space of functions with $(n_{1}-1)$ continuous derivatives (see Proposition \ref{consecutiveMS}). Some multi-spline spaces have proved to be of great interest for derivative sampling, where the goal is to reconstruct a signal from the samples of the function and of its first-order derivative. We should mention the well-known bicubic Hermite splines $(h_1,h_2)$, first introduced by Schoenberg and Lipow in \cite{Lipow1973}. They constitute a basis of $S_2+S_3$ with the shortest support and provide the direct interpolation formula
\begin{equation}
\forall f\in S_2+S_3, \quad \forall x\inR:f(x) = \sum_{k\inZ}\left(f(k)h_1(x-k)+f^{'}(k)h_2(x-k)\right),
\end{equation}
where $f^{'}=f^{(1)}$ is the derivative of $f$.
The excellent approximation capabilities and minimal-support property of the Hermite splines \cite{Fageot2020} give a strong incentive to investigate more general multi-spline spaces. The bicubic Hermite splines are the backbone of many computer-graphics applications and closely linked to B\'ezier curves \cite{Farouki2012,Uhlmann2016,Conti2015,Conti2016,Romani2020}. Schoenberg and Lipow also found two fundamental functions to reconstruct any function in $S_4+S_5$  from its samples and the samples of its first-order derivative. Nonetheless, those functions are not well-suited to practical applications since they are not compactly supported.
\\Building on top of an impressive body of work from various communities, we propose a systematic study of shortest bases for any multi-spline space. In particular, the main goal is to generalize the concept of B-splines to any multi-spline space.
\\
\\The paper is organized as follows: in Section \ref{sec:Formulation}, we formulate the problem in the framework of finitely generated shift-invariant spaces. We then state the properties that relevant generating functions should satisfy. In Section \ref{sec:Theory}, we show that the conditions imposed can only be met if the sum of the support of the generating functions is large enough. In Section \ref{sec:Multispline}, we present a method to construct shortest-support bases for any multi-spline space. This has important implications in practice, which we illustrate in Section \ref{sec:Applications} where we give practical examples to implement generalized sampling with the new set of functions, including interpolation, derivative sampling, and a new way to envision Bézier curves.

\section{Formulation of the Problem}
\label{sec:Formulation}
Let $\boldsymbol{\phi} = (\phi_1,\phi_2,\ldots,\phi_N)$  be a finite collection of functions in $ L_2(\mathbb{R})$, Lebesgue's space of square-integrable functions. The integer-shift-invariant subspace of $\LDR$ generated by $\boldsymbol{\phi}$ is denoted by $\shiftspace{\boldsymbol{\phi}}$ and is defined as
\begin{equation}
\shiftspace{\boldsymbol{\phi}} = \shiftspace{\phi_1} + \shiftspace{\phi_2} + \cdots+ \shiftspace{\phi_N}
,\end{equation}
where
\begin{equation}
\shiftspace{\phi_n}=\overline{\mathrm{Span}}\left(\{\phi_n(\cdot-k)\}_{k\in\mathbb{Z}}\right) \subseteq L_2(\mathbb{R}),\quad n=1,\ldots,N.
\end{equation}

We shall not restrict ourselves to multi-spline spaces for now and rather consider finitely generated integer-shift-invariant spaces. To formulate the problem, we recall three properties of $\boldsymbol{\phi}$ that have been imposed in previous works for practical applications. Multi-spline spaces will then naturally stand out as practical and important reconstruction spaces (Sections III and IV).
\subsection{Riesz Basis}
\begin{definition}
The set of functions $\{\phi_n(\cdot-k):k\inZ, n=1,\ldots ,N\} \subset L_2(\mathbb{R})$ is said to be a Riesz basis with bounds $A, B\inR$ with $0<A\leq B<+\infty$ if, for any vector of square-summable sequences $\boldsymbol{c} = (c_1,...,c_N)\inlDZN$, we have that
\begin{equation}
A \norm{\boldsymbol{c}}_{\ell_2}\leq \norm{\sum_{k\inZ}\boldsymbol{c}[k]^T \boldsymbol{\phi}(\cdot-k)}_{\LDR}\leq B \norm{\boldsymbol{c}}_{\ell_2}
,
\end{equation}
where $ \|{\boldsymbol{c}}\|_{\ell_2}=\left( \sum_{n=1}^N \|{c}_n\|_{\ell_2}^2\right)^{\frac{1}{2}}$,  $\boldsymbol{\phi} = (\phi_1,\phi_2,\ldots,\phi_N)$ and where A and B are the tightest constants.
\end{definition}
When this property is satisfied, we say that $\boldsymbol{\phi}$ generates a Riesz basis.
The Riesz-basis property guarantees that any $f\in \shiftspace{\boldsymbol{\phi}}$ has the unique and stable representation
(\cite{Christensen2016})
\begin{equation}
f(\cdot) = \sum_{k\inZ}\boldsymbol{c}[k]^T \boldsymbol{\phi}(\cdot-k) = \sum_{k\inZ}\sum_{n=1}^N c_n[k] \phi_n(\cdot-k)
.
\end{equation}
This property is well characterized in the Fourier domain via the Gramian matrix-valued function
\begin{equation}
\label{eq:gramian}
\hat{\boldsymbol{G}}(\omega) = \sum_{k\inZ}\hat{\boldsymbol{\phi}}(\omega+2k\pi)\hat{\boldsymbol{\phi}}(\omega+2k\pi)^{H} =  \sum_{k\inZ}\dotprod{\boldsymbol{\phi}}{\boldsymbol{\phi}^{T}(\cdot -k)}\ee^{-\jj\omega k}
,\end{equation}
where the inner product is defined as $\dotprod{f}{g} = \int_{\mathbb{R}}f(t)g^*(t)\dint t$, $\phantom{g}^*$ is the complex conjugate operator, and $^{H}$ is the conjugate transpose operator. Equality (\ref{eq:gramian}) follows from Poisson's formula applied to the sampling at the integers of the matrix-valued autocorrelation function $t\mapsto \dotprod{\boldsymbol{\phi}}{\boldsymbol{\phi}^{T}(\cdot -t)} = (\boldsymbol{\phi}*\boldsymbol{\phi}^{H\vee})(t)$\cite{Unser2014}. The Fourier equivalent of the Riesz-basis condition is \cite{Aldroubi1996}
\begin{equation}
\label{eigenvalueRB}
0<A^2=\essinf_{\omega \in [0,2\pi)}\lambda_{\min}(\omega)\leq \esssup_{\omega \in [0,2\pi)}\lambda_{\max}(\omega) = B^2<+\infty
,\end{equation}
where $\lambda_{\min}(\omega)$ and $\lambda_{\max}(\omega)$ are the smallest and largest eigenvalues of $\hat{\boldsymbol{G}}(\omega)$.
\subsection{Reproducing Polynomials}
\begin{definition}
The space $\shiftspace{\boldsymbol{\phi}}$ is said to reproduce polynomials of degree up to $M$ if, for all $m = 0,1,...,M$, there exist vector sequences $\boldsymbol{c}_m$ (not necessarily in $(\lDZ)^N$) such that
\footnote{for $m=0$, we use in (\ref{eq:polyrepro}) the convention that $x^m = 1$, including for $x=0$.}
\begin{equation}
\label{eq:polyrepro}
\forall x \inR,\quad x^m = \sum_{k\inZ}\boldsymbol{c}_m[k]^T \boldsymbol{\phi}(x-k)
.
\end{equation}
\end{definition}
Strang and Fix showed that the property of the reproduction of polynomials of degree up to $M$ is directly linked to the approximation power of the reconstruction space \cite{Strang1971}. More precisely, let
\begin{equation}
\mathrm{S}_h(\boldsymbol{\phi}) = \{f(\cdot/h):f\in \shiftspace{\boldsymbol{\phi}}\}
\end{equation}
be the $h$-dilate of $\shiftspace{\boldsymbol{\phi}}$. The space $\shiftspace{\V \phi}$ is said to have an approximation power of order $M$ if any sufficiently smooth and decaying function can be approached by an element of $\mathrm{S}_h(\V \phi)$ with an error decaying as $O(h^{M})$. The so called “Strang-Fix conditions" give sufficient conditions to have a space with an approximation power of order $M$ \cite{Fageot2020,DeBoor1998,Unser1997}. In particular, for compactly supported and integrable generating functions, it is sufficient to have the space $\shiftspace{\V \phi}$ reproduce polynomials of degree up to $(M-1)$. A straightforward implication is that the spline space $S_n$ has an approximation power of order $(n+1)$ since
\begin{enumerate}[label=(\roman*)]
\item it can reproduce polynomials of degree up to $n$;
\item it can be generated by the compactly supported function $\beta_+
^n$.
\end{enumerate}
The multi-spline space $S_{n_1}+\cdots+S_{n_N}$ inherits the highest approximation power of its spline spaces. Its approximation power is $(n_N+1)$, since $S_{n_N}\subset S_{n_1}+\cdots+S_{n_N}$.
\subsection{Compact Support}
The evaluation of $f\in \shiftspace{\boldsymbol{\phi}}$ at a given $x\inR$ from its discrete representation $\boldsymbol{c}\inlDZN$ requires a number of computations more or less proportional to the support size of $\boldsymbol{\phi}$. So, ideally, we want to minimize the support of $\boldsymbol{\phi}$ while maintaining a good approximation power \cite{Blu2001}. The support of a function $f\in \LDR$ is written as $\supp{f}=\overline{\{x\inR:f(x)\neq 0\}}$. If it is a compact subset of $\mathbb{R}$, then the support size is defined as $
\suppsize{f} =\int_{\mathbb{R}}\mathbbm{1}_{\supp{f}}(t)\dint t
$, where $\mathbbm{1}_{\supp{f}}$ is the indicator function of $\supp{f}$. For a finite collection of compactly supported functions $\V \phi = (\phi_{1},...,\phi_{N}),$ the natural extension for the support size is
\begin{equation}
\suppsize{\boldsymbol{\phi}} = \sum_{n=1}^N \suppsize{\phi_n}.
\end{equation}
In Section \ref{sec:Theory}, we present theoretical results that clarify the relation between the desired properties.
\section{Shortest Bases}
\label{sec:Theory}
For a single generator $\phi$ such that $\shiftspace{\phi}$ reproduces polynomials of degree up to $M$, Schoenberg stated  that $\suppsize{\phi}\geq M+1$ \cite{Schoenberg1973}. The result was proved in \cite{Aziznejad2019} for $N=2$. We now extend the proof to any $N\inN\setminus \{0\}$.
\begin{theorem}[Minimal support]
\label{MinimalSupport}
If $\shiftspace{\boldsymbol{\phi}} = \shiftspace{\phi_1,\phi_2,\ldots,\phi_N}$ reproduces polynomials of degree up to $M$, then $\suppsize{\boldsymbol{\phi}} \geq M+1$. In addition, if there is equality, then
\begin{equation}
\label{minimalsupportcharacterization}
    \sum_{k\inZ} \sum_{n=1}^N  \mathbbm{1}_{\supp{\phi_n}}(x+k) = \suppsize{\boldsymbol{\phi}} \quad \text{for almost every } x\inR.
\end{equation}
\end{theorem}
\begin{proof}
   If $\boldsymbol{\phi}$ is not compactly supported, then the inequality is clear. Now, we can assume that $\boldsymbol{\phi}$ is compactly supported.
   This implies that, for any $x\inR$, the sum $\sum_{k\inZ}\boldsymbol{c}[k]^T \boldsymbol{\phi}(x-k) = \sum_{k\inZ} \sum_{n=1}^N c_n[k]\phi_n(x-k)$ has only a finite number of nonzero terms that are identified by the set
   \begin{equation}
   \Lambda(x)= \left\{(n,k)\in \{1,\ldots,N\}\times \mathbb{Z} : \quad x \in \supp{\phi_n(\cdot-k)}\right \}
   ,
   \end{equation}
   
   and its cardinality
    \begin{equation}
    \label{eq:cardinality}
    \lambda(x) = \#(\Lambda(x))=\sum_{k\inZ} \sum_{n=1}^N  \mathbbm{1}_{\supp{\phi_n}}(x+k) \inN
    .
    \end{equation}
    Equation (\ref{eq:cardinality})  follows from the fact that $\mathbbm{1}_{\supp{\phi_n}}(x+k)$ is 1 if and only if $(n,k)\in \Lambda (x)$ and 0 otherwise.
   The function $x\mapsto \lambda(x)$ is 1-periodic and bounded because $\supp{\phi_n}$ are compact subsets of $\mathbb{R}$. Its average over one period reads (note that the sums are in fact all finite)
   \begin{align}
   \label{eq:fubini}
       \overline{\lambda} &= \int_{0}^1 \sum_{n=1}^N \sum_{k\inZ} \mathbbm{1}_{\supp{\phi_n}}(x+k) \dint x
      =  \sum_{n=1}^N \sum_{k\inZ} \int_{0}^1 \mathbbm{1}_{\supp{\phi_n}}(x+k) \dint x \\
       &= \sum_{n=1}^N \int_{-\infty}^{\infty} \mathbbm{1}_{\supp{\phi_n}}(x) \dint x = \suppsize{\boldsymbol{\phi}}
  , \end{align}
where we applied Fubini's Theorem in (\ref{eq:fubini}).
   Because $\lambda$ is bounded and takes values in $\mathbb{N}$, it only takes a finite number of values. Consequently, there exists a set $\mathrm{A}\subset [0,1]$ of nonzero measure such that $\lambda$ is constant on A and no greater than its average, as in
\begin{equation}
       \forall x\in A : \quad \lambda(x) = \lambda_{A} \leq \overline{\lambda} = \suppsize{\boldsymbol{\phi}}
    .
\end{equation}
The function $\#(\Lambda)$ restricted to A is constant, but this does not imply that $\Lambda$ is constant on A. Noting that A is bounded and that the $\phi_n$ are compactly supported, the image of A under $\Lambda$, denoted by $\Lambda(A)$, is a finite set. Therefore, there exists $\mathrm{B}\subset \mathrm{A}\subset[0,1]$ of nonzero measure such that $\Lambda$ is constant on B. This means that the set $\shiftspace{\boldsymbol{\phi}}_{|B}$ of functions of $\shiftspace{\boldsymbol{\phi}}$ restricted to $B$ is spanned by $\lambda_A$ functions $(\phi_n(\cdot-k))_{(n,k)\in \Lambda(B)}$.
   \newline
   Moreover, due to the reproducing property, the polynomials of degree up to $M$ restricted to B form a linear subspace of $\shiftspace{\boldsymbol{\phi}}_{|B}$ whose dimension is $(M+1)$, because B is infinite. Then, we must have that $\lambda_A\geq M+1$ and, since $\lambda_A\leq \suppsize{\boldsymbol{\phi}}$, we deduce the announced bound $\suppsize{\V \phi}\geq M+1$.
\newline If $\lambda$ is not $a.e.$ constant, then $\mathrm{A}$ can be chosen so that $\lambda_A<\overline{\lambda} = \suppsize{\boldsymbol{\phi}}$ and $\shiftspace{\boldsymbol{\phi}}_{|B}$ is spanned by fewer than $\suppsize{\boldsymbol{\phi}}$ functions. The reproduction property implies that $\suppsize{\boldsymbol{\phi}} > M+1$. This means that the equality  $\suppsize{\boldsymbol{\phi}} = M+1$ is possible only if $\lambda$ is $a.e.$ constant.
\end{proof}
Following Theorem \ref{MinimalSupport}, we can introduce the central notion of shortest-support basis.
\begin{definition}
A collection of functions $\boldsymbol{\phi}\in (\LDR)^N$ is said to be a shortest-support basis of degree $M$ if $\shiftspace{\boldsymbol{\phi}}$ reproduces polynomials of degree up to $M$ with the shortest support, \textit{i.e.} with $\suppsize{\boldsymbol{\phi}} = M+1$.
\end{definition}

The qualifier of basis comes from Theorem \ref{MinimalSupport_RB}.
\begin{theorem}[Shortest support and Riesz basis]
\label{MinimalSupport_RB}
Any shortest basis generates a Riesz basis.
\end{theorem}
Before proving the theorem, we define the $k$th slice of any function $f$ as
\begin{equation}
\forall x \inR : \quad \mathrm{S}_k\{f\}(x) = \begin{cases}f(x+k),&x\in[0,1)\\0,& \text{otherwise,}\end{cases}
\end{equation}
and the set of nonzero slices of all the generating functions as
\begin{equation}
    \mathcal{T}(\boldsymbol{\phi}) = \{\mathrm{S}_k\{\phi_n\} \textnormal{}: \mathrm{S}_k\{\phi_n\} \not\equiv 0 \textnormal{ and }k\inZ, n=1,...,N\}
.\end{equation}
The proof will also invoke Lemma \ref{LemmaRB}.
\begin{lemma}
\label{LemmaRB}
    Let $\boldsymbol{\phi}\in (\LDR)^{N}$ be compactly supported. If $\mathcal{T}(\boldsymbol{\phi})$ is a set of linearly independent functions, then $\boldsymbol{\phi}$ generates a Riesz basis.
\end{lemma}
\begin{proof}
The generating functions can be expressed in terms of their slices as $\phi_n(x) = \sum_{k\inZ}\mathrm{S}_k\{\phi_n\}(x - k)$. The Riesz-basis property is best characterized in the Fourier domain with the Gramian matrix (note that, $\boldsymbol{\phi}$ being compactly supported, all the sums are in fact finite), which leads to
 \begin{align}
     (\hat{\boldsymbol{G}}(\omega))_{mn} &= \sum_{q\inZ}\dotprod{\phi_m}{\phi_n(\cdot -q)}\ee^{-\jj\omega q}  \nonumber \\
     &=\sum_{q\inZ}\sum_{k_1\inZ}\sum_{k_2 \inZ}\dotprod{\mathrm{S}_{k_1}\{\phi_m\}}{\mathrm{S}_{k_2}\{\phi_n\}(\cdot - q - (k_2-k_1))}\ee^{-\jj\omega q}  \nonumber\\
     &=\sum_{k_1\inZ}\sum_{k_2 \inZ}\dotprod{\mathrm{S}_{k_1}\{\phi_m\}}{\mathrm{S}_{k_2}\{\phi_n\}}\ee^{\jj\omega(k_2-k_1)}&&\text{if $q\neq (k_1-k_2)$, the inner product vanishes}  \nonumber \\
     &=\dotprod{\sum_{k_1 \inZ}\mathrm{S}_{k_1}\{\phi_m\}\ee^{-\jj\omega k_1}}{\sum_{k_2 \inZ}\mathrm{S}_{k_2}\{\phi_n\}\ee^{-\jj\omega k_2}}  \nonumber \\
     &=\dotprod{\tilde{\phi}_m(\omega,\cdot)}{\tilde{\phi}_n(\omega,\cdot)},
 \end{align}
 where $\Tilde{\phi_n}(\omega,\cdot)$ is the finite weighted sum of slices
\begin{equation}
     \Tilde{\phi_n}(\omega,x) = \sum_{k \inZ}\mathrm{S}_{k}\{\phi_n\}(x)e^{-j\omega k}
 .
 \end{equation}
If, now, $\mathcal{T}(\boldsymbol{\phi})$ is a set of linearly independent functions, then, for any $\omega \inR$, the functions $(\Tilde{\phi_n}(\omega,\cdot))_{n=1,...,N}$ are linearly independent because the sums are finite. This means that $\hat{\boldsymbol{G}}(\omega)$ is the Gramian matrix of a linearly independent family of functions, which is known to be equivalent to $\det \hat{\boldsymbol{G}}(\omega) > 0$.
In addition $g:\omega \mapsto \det(\hat{\boldsymbol{G}}(\omega))$ is a finite  weighted sum of $\ee^{\jj\omega k}$ since $\boldsymbol{\phi}$ is compactly supported. It is therefore continuous and $2\pi$-periodic. The image of $[0,2\pi]$ under $g$ is therefore a closed interval such that
\begin{equation}
0<\essinf_{\omega\in [0,2\pi]}
\det(\hat{\boldsymbol{G}}(\omega))=\min_{\omega\in [0,2\pi]}
\det(\hat{\boldsymbol{G}}(\omega))<\esssup_{\omega\in [0,2\pi]}
\det(\hat{\boldsymbol{G}}(\omega))=\max_{\omega\in [0,2\pi]}
\det(\hat{\boldsymbol{G}}(\omega))<+\infty.
\end{equation}
 Noting that $\det(\hat{\boldsymbol{G}}(\omega))$ is the product of the eigenvalues of $\hat{\boldsymbol{G}}(\omega)$, Condition (\ref{eigenvalueRB}) is satisfied, which means that $\boldsymbol{\phi}$ is a Riesz basis.
\end{proof}
Note that the converse of Lemma \ref{LemmaRB} is not necessarily true. For a counterexample, consider the function in (\ref{eq:counterexample}) made of two side-by-side rectangles of different height, so that
\begin{equation}
\label{eq:counterexample}
\forall x \inR:\phi(x) = \begin{cases}1,& x\in[0,1)\\ \alpha,& x \in [1,2)\\0, & \text{otherwise.}\end{cases}
\end{equation}
In this case, with a single generator, the Gramian matrix is just a scalar and reads $\hat{g}(\omega) = (1+\alpha^2)+2\alpha \cos \omega,$ which verifies, for any $\omega\inR$, that
\begin{equation}
(1-\abs{\alpha})^2\leq\abs{\hat{g}(\omega)}\leq (1+\abs{\alpha})^2.
\end{equation}
So for $\abs{\alpha} \neq 1$, $\phi$ is a Riesz basis with bound $A=(1-\abs{\alpha})$ and $B=(1+\abs{\alpha})$. Yet, $\mathcal{T}(\boldsymbol{\phi})$ is clearly not a set of linearly independent functions since the second slice is a scaled version of the first one.
 For a more practical counterexample, see \cite[Proposition 2.2.]{Antonelli2014}. \\
 \begin{lemma}
 \label{lm:sliceslinearlyindep}
 Let $\V \phi \in (\LDR)^{N}$. If $\V \phi$ is a shortest-support basis, then $\mathcal{T}(\boldsymbol{\phi})$ is a set of linearly independent functions.
 \end{lemma}
 \begin{proof}
It is equivalent to prove the contrapositive of the lemma, which states that  if $\mathcal{T}(\boldsymbol{\phi})$ is not a set of linearly independent functions, then $\V \phi$ is not a shortest-support basis. To that end, suppose that $\mathcal{T}(\boldsymbol{\phi})$ is not a set of linearly independent functions. This means that one can find a slice, say $\mathrm{S}_{k_0}\{\phi_{q_0}\}$, that depends linearly on the others. Now, consider the integer-shift-invariant space generated by the set of functions $\mathcal{T}(\boldsymbol{\phi}) \backslash \{\mathrm{S}_{k_0}\{\phi_{q_0}\}\}$. Note that the new generating functions differ now both in size (support size of at most 1) and in number (possibly greater than $N$). On one hand, the new integer-shift-invariant space  is larger than the initial space and, in particular, is still able to reproduce polynomials of degree up to $M$. On the other hand, the sum of the support size of the generating functions is smaller than $\suppsize{\boldsymbol{\phi}}$ because a nonzero slice was removed. So, $\boldsymbol{\phi}$ cannot be of minimal support.
 \end{proof}
We can now prove Theorem \ref{MinimalSupport_RB}.
\begin{proof}[Proof of Theorem \ref{MinimalSupport_RB}]
Let $\V \phi \in (\LDR)^{N}$ be compactly supported. By contraposition, if it is not a Riesz basis, then $\mathcal{T}(\boldsymbol{\phi})$ is not a set of linearly independent functions (Lemma \ref{LemmaRB}). Then, by Lemma \ref{lm:sliceslinearlyindep}, $\boldsymbol{\phi}$ cannot be of minimal support.
\end{proof}
To conclude this section, we present two results for finitely generated integer-shift-invariant spaces in preparation to a characterization of multi-spline spaces (Theorem \ref{prop:MN}). The unit sample sequence is written $\delta[\cdot]$ and is defined by $\delta [k] =  \begin{cases}1, & k=0\\0, & k\neq 0\end{cases}$, and its matrix version $\V \delta_{N\times N}$ is defined by $(\V \delta_{N\times N})_{pq}[\cdot] = \begin{cases}\delta[\cdot], & p=q\\0, & p\neq q\end{cases}$.
\begin{lemma}
\label{lm:convo}
Let $N,M\inN$, $\V C \in (\mathbb{R}^{\mathbb{Z}})^{N\times M}$, $\V B \in (\mathbb{R}^{\mathbb{Z}})^{M\times N}$. If the sequence of matrices $\V B$ is compactly supported and $\V C*\V B = \V \delta_{N\times N}$, then $M\geq N$.
\end{lemma}
\begin{proof}
There exists $s\inN$ such that $\supp{\V B}\subset \{-s,...,s\}\subset \mathbb{N}$. The behavior of $\V C[k]$ when $\abs{k}\rightarrow \infty$ is not known, and it is easier to work with the truncated version $\V C_{m}= \mathbbm{1}_{\{-s,...,ms\}} \times \V C $, where $m\inN$ is a large enough integer $m>2N+1$. The sequence of matrices $\V C_{m}*\V B$ is compactly supported and satisfies $\supp{\V C_{m}*\V B}\subset \{-2s,...,(m+1)s\}$. Following the properties of convolution of compact sequences, we have, for any $k=0,...,(m-1)s$, that $\V C_{m}*\V B[k]=\V C*\V B[k] = \V \delta_{N\times N}[k]$. Therefore, one can write that
\begin{equation}
\V C_{m}*\V B= \V \delta_{N\times N}+\sum_{\substack{-2s\leq k < 0\\(m-1)s+1 \leq k \leq (m+1)s}}\mathbf{M}_{k}\delta[\cdot-k],
\end{equation}
where $\mathbf{M}_{k}\in \mathbb{R}^{N\times N}$ are matrices that account for the fact that $\V C_{m}$ is a truncated version of $\V C$.
This then translates into the following z-transform matrix relation (note that all sequences are compactly supported so the z-transforms are well defined)
\begin{equation}
\hat{\V C}_{m}(z)\hat{\V B}(z) = \mathbf{I}_{N\times N}+\sum_{\tiny \substack{-2s\leq k < 0\\(m-1)s < k \leq (m+1)s}}z^{-k} \mathbf{M}_{k} =\mathbf{M}_{N\times N}+\sum_{k=-2s}^{-1}z^{-k} \mathbf{M}_{k}+\sum_{k=(m-1)s+1}^{(m+1)s}z^{-k} \mathbf{M}_{k}= z^{-2s} \V A(z)
,
\end{equation}
where $\V A(z)$ can be decomposed as
\begin{equation}
\V A(z) = z^{2s} \mathbf{I}_{N\times N}+\V P(z) +z^{(m+1)s+1}\V Q(z),
\end{equation}
 where $\V P(z)$ and $\V Q(z)$ are polynomial matrices of degree $(2s-1)$. The determinant of $\V A(z)$ can be expressed in terms of the columns of $\mathbf{I}_{N\times N}, \V P(z)$, and $\V Q(z)$ (denoted respectively $\mathbf{e}_k,\V p_k(z)$, and $\V q_k(z)$), so that
\begin{equation}
z\mapsto \det A(z) = \det(z^{2s} \mathbf{e}_1+\V p_1(z)+z^{(m+1)s} \V q_1(z),\ldots,z^{2s} \mathbf{e}_N+\V p_N(z)+z^{(m+1)s} \V q_N(z))
.
\end{equation}
Knowing that the determinant is $n$-linear with respect to the columns, $z\mapsto \det \V A(z)$ is a polynomial function of degree at most $(m+3)sN$. We now want to prove that it cannot be identically zero. To that end, we expand the determinant with respect to the columns and find that there is a unique term of the form $\lambda z^{2sN}$. It is obtained by picking for $k=1,\ldots,N$ the column $\mathbf{e}_k z^{2s}$. The coefficient in front of $z^{2sN}$ is therefore $\det(\mathbf{e}_1,\ldots, \mathbf{e}_N) = 1 \neq 0$. Indeed, for other combinations of columns in the expansion, we would have that
\begin{itemize}
\item if at least one column of the form $z^{(m+1)s}\V q_k(z)$ is chosen, then it results in a term of degree at least $(m+1)s>(2N+2)s>2sN$;
\item else, at least one column of the form $\V p_k(z)$ is chosen. Since the degree of $\V p_k(z)$ is lower than $2s$, the resulting term in the expansion has a degree lower than $2sN$.
\end{itemize}
In the end, we proved that $z\mapsto \det \V A(z)$ cannot be identically zero. Therefore, there exists $z_{0}\inR$ so that $\mathrm{rank}(\V A (z_{0}))=N$. It implies that $N = \mathrm{rank}(\hat{\V C}_{m}(z_{0})\hat{\V B}(z_{0}))\leq \min(\mathrm{rank}(\hat{\V C_{m}}(z_{0})),\mathrm{rank}(\hat{\V B}(z_{0})))\leq \min(M,N)\leq M$.
\end{proof}
\begin{lemma}
\label{lm:MinimalGeneratingFunctions}
Let $\V \psi \in (\LDR)^M$ and $\V \eta \inLDRN$ be two collections of compactly supported functions that are able to reproduce each other (the reproducing sequences might not be in $\lDZ$). If $\V\eta$ is a shortest-support basis, then $M\geq N$.
\end{lemma}
\begin{proof}
By hypothesis, there exist vector sequences $\V c_p \in (\mathbb{R}^{\mathbb{Z})^M}$ such that $\eta_p = \sum_{k\inZ}\V c_p[k]^T\boldsymbol{\psi}(\cdot-k) = \V {c}_p^T*\boldsymbol{\psi}$, which reads in matrix form
\begin{equation}
\V\eta = \V C*\V \psi, \quad \V C \in (\mathbb{R}^{\mathbb{Z}})^{N\times M}
.
\end{equation}
Similarly, one can write that
\begin{equation}
\V\psi= \V B*\V \eta, \quad \V B \in (\mathbb{R}^{\mathbb{Z}})^{M\times N}
.
\end{equation}
From Lemma \ref{lm:sliceslinearlyindep}, we know that the nonzero slices of $\V \eta$ are linearly independent (shortest-support basis). This implies that, to generate the compactly supported function $\V \psi$, the sequence of matrices $\V B$ must be compactly supported as well since the only way to generate the zero function on a segment for $\V \eta$ is to set the active coefficient of $\V B$ to $0$. Now, one can mix the equations and find that
\begin{equation}
\V\eta = \V C*(\V B*\V \eta) = (\V C*\V B)*\V \eta
.
\end{equation}
The associativity of the convolution operations is justified by the fact that both $\V \eta$ and $\V B$ are compactly supported, meaning that, for a given argument $x$, all sums are finite.
Because the slices of $\V \eta$ are linearly independent, $\V \eta$ can reproduce itself in a unique way, which gives
\begin{equation}
\label{eq:convolemma}
\V C*\V B= \V \delta_{N\times N},
\end{equation}
We can now conclude that $M\geq N$ with Lemma \ref{lm:convo}.
\end{proof}

\section{Multi-Spline Shortest Bases}
\label{sec:Multispline}
With a single generator, the unique shortest basis of degree $n\inN$ (up to a scaling and  a shift operation) is the B-spline of degree $n$, which is a generator of $S_n$. For multiple generators, it is natural to consider spaces generated by a finite number of B-splines $\boldsymbol{\beta}_{\mathbf{n}} = \left(\beta_+^{n_1},..., \beta_+^{n_N}\right)$, where $\mathbf{n} = (n_1,\ldots,n_N)$ and $n_1<\ldots<n_N$. In this way, the reproducing and approximation properties are inherited from the higher-degree spline $\beta_+^{n_N}$. Yet, multi-spline spaces are not generated optimally by the classical B-splines.
\begin{proposition}
Let $N\inN\setminus{\{0\}}$ and $\mathbf{n} = (n_1,\ldots,n_N)$ with $n_1<\cdots <n_N\inN$. If $N>1$, then $\boldsymbol{\beta}_{\mathbf{n}} = \left(\beta_+^{n_1},..., \beta_+^{n_N}\right)$ is neither a shortest-support basis nor a Riesz basis.
\end{proposition}
\begin{proof}
\begin{itemize}[leftmargin=*]
\item The space $\shiftspace{\boldsymbol{\beta}_{\mathbf{n}}}$ can reproduce polynomials of degree at most $n_N$ due to the inclusion $\shiftspace{\beta_+^{n_N}}\subset \shiftspace{\V \phi}$. Moreover, the sum of the support of $\V \beta_{\mathbf{n}}$ is $\sum_{m=1}^N (n_m+1)>n_N+1$, which shows that the basis is not a shortest-support one.
\item From the proof of Lemma \ref{LemmaRB}, the Gramian matrix can be written
 \begin{align}
     (\hat{\boldsymbol{G}}(\omega))_{pq}=\dotprod{\Tilde{\beta}^{n_p}(\omega,\cdot)}{\Tilde{\beta}^{n_q}(\omega,\cdot)},
 \end{align}
 where $\Tilde{\beta}^{n_p}(\omega,\cdot)$ is the finite weighted sum of slices
 \begin{equation}
     \Tilde{\beta}^{n_p}(\omega,x) = \sum_{k \inZ}\mathrm{S}_{k}\{\beta^{n_p}\}(x)\ee^{-\jj\omega k}
 .\end{equation}
It is known that $\beta_+^{n_p}$ satisfies the partition of unity, meaning that, for any $x\inR, \sum_{k\inZ}\beta_+^{n_p}(x-k) = 1$. In terms of slices, it means that $\Tilde{\beta}^{n_p}(0,x) =\sum_{k \inZ}\mathrm{S}_{k}\{\beta^{n_p}\} (x)= \mathbbm{1}_{[0,1)}(x)$. The functions $(\Tilde{\beta}^{n_p}(0,\cdot))_{p=1,...,N}$ are therefore not linearly independent (because they are equal) and $\det \hat{\boldsymbol{G}}(0) = 0$. As stated in the proof of Lemma \ref{LemmaRB}, $\omega \mapsto \det \hat{\boldsymbol{G}}(\omega)$ is a continuous function (because the B-splines are compactly supported), meaning that
\begin{equation}
\essinf_{\omega \in [0,2\pi]}{\det \hat{\boldsymbol{G}}(\omega)} = \min_{\omega \in [0,2\pi]}{\det \hat{\boldsymbol{G}}}(\omega) = 0.
\end{equation}
Following (\ref{eigenvalueRB}), $\boldsymbol{\beta}_{\mathbf{n}}$ cannot be a Riesz basis.
\end{itemize}
\end{proof}
For $N>1$, only few shortest bases are known, with the most prominent being the Hermite splines presented by Lipow and Schoenberg \cite{Lipow1973}. They are solution of the direct interpolation problem
\begin{equation}
\label{eq:Hermite}
\text{find } \eta_p\in S_{n,N}: \eta_p^{(\nu)}(k) = \begin{cases}1,& \text{if } \nu=p \text{ and } k=0,\\0,&\text{otherwise,}\end{cases}
\end{equation}
with $k\inZ, \quad \nu,p = 0,...,(N-1)$, and $S_{n,N} = S_{n}+\cdots+S_{n+N-1}$.
The function $\eta_p$ has all its derivatives set to zero at the integers, except for the $p$th derivative that is one at zero. The multi-spline space must be chosen so that $\eta_p$ is sufficiently differentiable, yielding the condition $n\geq N$. When $n=N$, shortest-support functions were found (the Hermite splines, see plots \cite{Ranirina2019} for instance) but, unfortunately, in a higher-order approximation space, {\it i.e.} for $n>N$, the functions are not compactly supported anymore. For instance, for derivative sampling (interpolate $f$ and $f^{'}$), the smaller order of approximation solution ($N=2$) is given by the cubic Hermite-spline generators of $S_2+S_3$.
\subsection{Consecutive Multi-Spline Spaces}
The derivatives up to order $(n-1)$ of a compactly-supported spline of degree $n$ must vanish on the edges of the support. This constraint cannot be satisfied if the function is too short. In particular, the shortest nonzero function of $S_n$ has a support size of $(n+1)$ and, interestingly, it is precisely the B-spline of degree $n$. In the special case of a consecutive multi-spline space $S_{n,N}=S_{n}+S_{n+1}+\cdots+S_{n+N-1}$, this result can be directly extended. To that end, we define the space
\begin{equation}
P_{m}^{m'} = \{p \in C^{m'}(\mathbb{R}): \text{$p$ is a polynomial of degree $m$ on each } [k,k+1), k\inZ\}
.
\end{equation}
Note that the space $P_{m}^{m'}$ can be viewed as a spline space with knots of multiplicity $(m-m'-1)$ (\cite[Section 5.11]{Lachance1990}). In our setting with simple knots, $P_{m}^{m'}$ is rather regarded as multi-spline space (Proposition \ref{consecutiveMS}).
\begin{proposition}
\label{consecutiveMS}
    Let $n,N>0$. Then
    $
    S_{n,N}= P^{n-1}_{n+N-1}$.
\end{proposition}
\begin{proof}The definition of a spline of degree $n$ implies that, for $q=0,...,N-1$, we have that $S_{n+q}\subset P^{n-1}_{n+N-1}$, from which we deduce that $S_{n,N}\subset P^{n-1}_{n+N-1}$.\\
The other inclusion is proven by induction over $N$, with the induction hypothesis
\begin{equation}
H_{N}:\forall n\inN, P^{n-1}_{n+N-1}\subset S_{n,N}.
\end{equation}
\begin{itemize}
\item For $N=1$ and any $n\inN\setminus{\{0\}}$, the result is directly given by the definition of $S_{n,1}=S_{n} = P^{n-1}_{n}$.
\item Suppose that $H_{N}$ holds for $N\inN^*$. Let $p\in P^{n-1}_{n+N}$. We have that $p^{(n-1)}\in P^{0}_{N+1}$ and, consequently, $p^{(n)}$ is a piecewise polynomial function with finite jumps at the knots. There exists $f_{0}\in S_{0}$ that has the same jumps on the knots as $p^{(n)}$. Then, $(p^{(n)}-f_{0})$ is continuous on the integers, which implies that $(p^{(n)}-f_{0})\in P^{0}_{N}$. The induction hypothesis guarantees that $(p^{(n)}-f_{0})\in S_{1,N}$ and, therefore, that $p^{(n)}\in S_{0,N+1}$. After $n$ integrations, we finally have that $p\in S_{n,N+1}$, which concludes the induction step and the proof.
\end{itemize}
\end{proof}
For a given $L\inN$, the space of functions in $P^{m'}_m$ that are supported in $[0,L]$ is a vector space of the known finite dimension \cite{Alfeld1985}
\begin{equation}
\dim(\{p\in P_m^{m'}:\supp{p}\subset [0,L]\}) = ((m-m')L-(m'+1))_+
,
\end{equation}
where $x_+=\max(0,x)$.
Indeed, any $p\in P_m^{m'}$ supported in $[0,L]$ is uniquely defined by $L$ pieces that are polynomials of degree $m$. So, $L \times(m+1)$ coefficients have to be set. The smoothness constraints imply that the pieces cannot be set independently. On the first interval $[0,1)$, the $(m+1)$ coefficients must be chosen so that $p^{(0)},...,p^{(m')}(0) = 0$, which leaves $(m-m')$ degrees of freedom. For the next interval, $(m+1)$ new coefficients have to be set but the values $p^{(0)}(1),...,p^{(m')}(1)$ are already fixed, giving only $(m-m')$ new degrees of freedom. We see that each interval provides $(m-m')$ extra degrees of freedom. In the end, there remain $LN$ degrees of freedom. Now, to enforce that $p\in P_{m}^{m'}$, we must have that $p^{(0)}(L),...,p^{(m')}(L)=0$. The total number of degrees of freedom gives the announced dimension $((m-m')L-(m'+1))_+$.
\begin{corollary}
\label{DimensionCompact}
Let $n,N,L\inN$. The set of functions of $S_{n,N}$ that have their support in $[0,L]$  is a vector space of dimension $(LN-n)_+=\max(0,LN-n)$.
\end{corollary}

\begin{corollary}
Let the Euclidean division of $n$ by $N$ be written as $n = pN+r$. Then, the shortest-support nonzero functions of $S_{n,N}$ have a support size of $(p+1)$. Moreover, the set $\{ f \in S_{n,N}:\supp{f}\subset [0,p+1] \}$ is a vector space of dimension $(N-r)$.
\end{corollary}
\begin{proof}
The set of functions of $S_{n,N}$ that have their support in $[0,L]$ is a vector space of dimension $(LN-n)_+$ (Corollary \ref{DimensionCompact}). To find at least one non-vanishing function in the vector space, its dimension must be greater than one meaning that $(LN-n)\geq 1 \Leftrightarrow L\geq (n+1)/N = p+(r+1)/N$. Knowing that $L\in \mathbb{N}$ and $r<N$, we conclude that one must have that $L = (p+1)$ to find a nonzero compactly supported function. In this case, the dimension reads $((p+1)N-n)=(N+pN-n) = (N-r)$.
\end{proof}
With a single generator, the shortest-support basis is provided by the shortest function. In a consecutive multi-spline space, one would ideally take $(N-r)$ functions of size $(p+1)$ (the shortest) and complete with $r$ functions of size $(p+2)$. This would result in $N$ functions with a total support size of $(N-r)(p+1)+r(p+2) = Np+r+N = n+N=n_N+1$, which is the objective for a shortest-support basis.
For nonconsecutive multi-spline spaces, similar results should exist, but in a more complicated form.
\subsection{Existence and Construction of mB-Splines}
We say that a finite collection $\V \phi$ of multi-spline functions is an mB-spline of degree $\mathbf{n} = (n_1,\ldots,n_N)$ with $n_1<\cdots <n_N\inN$, if it is a shortest-support basis of the space $S_{\mathbf{n}}$. This is the natural extension of B-splines. Similar to the latter, mB-splines can be constructed recursively for any multi-spline space. Indeed, two basic transformations (the “increment step" and the “insertion step") allow one to convert a shortest-support basis of a given space into a shortest-support basis of a different space.
To simplify the explanation, we say that the collection $\boldsymbol{\phi} = (\phi_1,...,\phi_N) \inLDRN$ of compactly supported functions is standardized if, for $n=1,\ldots,N$, we have that
\begin{enumerate}[label=(\roman*)]
\item $\int_\mathbb{R}\phi_n(t)\dint t \in \{0,1\}$,
\item $\inf \{t\inR\colon \phi_n(t)\neq 0\}\in [0,1)$.
\end{enumerate}

The second condition implies that the generating functions are causal, {\it i.e.} $\phi_{n}(t<0) = 0$.
Note that any $\boldsymbol{\phi}$ compactly supported can be standardized without altering $\shiftspace{\boldsymbol{\phi}}$.
\subsubsection{Increment Step}
The B-splines $\beta_+^{n+1}$ can be constructed recursively by noting that
\begin{equation}
\beta_+^{n+1}(x) = \Delta \left \{ \int_{-\infty}^x \beta_+^{n}(t)\dint t\right \}
,
\end{equation}
where $\Delta$ is the finite difference operator $\Delta \{f\}(x) =(f(x) - f(x-1))$. The integration increases the polynomial degree, along with the smoothness at the knots (Step 1), while $\Delta$ ultimately returns a compactly supported function (Step 2). For multiple generating functions, a similar two-step recursive approach is proposed. The general process is mathematically detailed below, while an intuitive example is proposed in Figure \ref{incremementStep}.
\begin{figure}[t!]
\begin{minipage}{1.0\linewidth}
  \centering
  \centerline{\includegraphics[width=\linewidth]{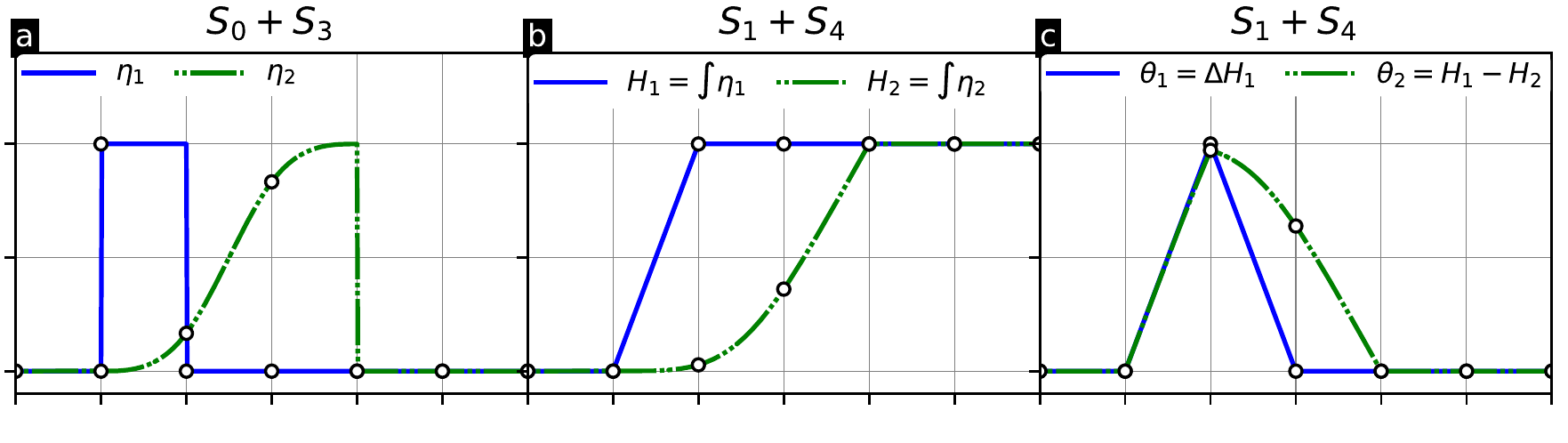}}
  \caption{Increment step that yields a shortest-support basis of $S_1+S_4$ starting from $S_0+S_3$. (a) A shortest-support basis $(\eta_1,\eta_2)$ for $S_0+S_3$ ($\suppsize{\boldsymbol{\eta}}=4$). (b) The integration of  $\eta_1$ and $\eta_2$ results in two generators of $S_1+S_4$, $H_1$ and $H_2$. (c) To get compactly supported functions with the same generating properties, we choose $\theta_1 = \Delta H_1$ and $\theta_2=(H_1-H_2)$. We found a shortest support-basis of $S_1+S_4$ ($\suppsize{\boldsymbol{\theta}}=5$).}
  \label{incremementStep} \medskip
\end{minipage}
\end{figure}

Suppose $\boldsymbol{\eta} = (\eta_1,...,\eta_N)\inLDRN$ is an mB-spline of $S_{n_1}+\cdots+S_{n_N}$. The goal is to find an mB-spline of $S_{n_1+1}+\cdots+S_{n_N+1}$. It will be a generator with a support size of $(n_N+2)$, able to reproduce the B-splines of degree $n_1+1,...,n_N+1$.
\subsubsection*{{\it Integration}}
The collection of functions $\boldsymbol{\eta}$ is able to reproduce the B-splines of degree $n_1,\ldots,n_N$, that is, for any $s \in \{1,...,N\}$ there exists a vector sequence $\boldsymbol{c}^s = (c_1^s,...,c_N^s)$ (not necessarily in $(\lDZ)^N$) so that
\begin{equation}
\label{reprobeta}
\forall x\inR:\beta_+^{n_s}(x) = \sum_{k\inZ}\boldsymbol{c}^s[k]^T\boldsymbol{\eta}(x-k)
.\end{equation}
To justify the calculations to come, we assume that
\begin{equation*}
c_1^s,...,c_N^s \hbox{ are causal sequences, {\it i.e.}, }  c_n^s[k] = 0 \text{ for any } k<0. \qquad (A_{\mathbf{n}}) 
\end{equation*}
The assumption $(A_{\mathbf{n}})$ is not overly restrictive because it will hold for the starting basis of our algorithm and then be preserved by the construction process. In the end, all the bases constructed will be able to reproduce the B-splines with causal sequences.
Let $\boldsymbol{H} = (H_1,...,H_N)$ be defined as
\begin{equation}
\boldsymbol{H}(x) = \int_{-\infty}^x \boldsymbol{\eta}(t) \dint t
.
\end{equation}
The integration of equation (\ref{reprobeta}), followed by the application of the operator $\Delta$, yields
\begin{align}
\beta_+^{n_s+1}(x)&= \Delta \left\{\sum_{k\inZ} \boldsymbol{c}^s[k]^T \boldsymbol{H}(x-k)
\right\}= \sum_{k\inZ} \boldsymbol{c}^s[k]^T \Delta \{\boldsymbol{H}\}(x-k) \nonumber \\&= \sum_{k\inZ} \boldsymbol{c}^s[k]^T (\boldsymbol{H}(x-k)-\boldsymbol{H}(x-1-k)) =\sum_{k\inZ} (\boldsymbol{c}^s[k]^T -\boldsymbol{c}^s[k-1]^T)\boldsymbol{H}(x-k)
.\label{reproducecausal}\end{align}
The assumption that $c_1^s,...,c_N^s$ are causal and the fact that $\boldsymbol{H}$ is also causal (because $\boldsymbol{\eta}$ is compactly supported and standardized) implies that, for any $x\inR$, the sums in (\ref{reproducecausal}) have a finite number of nonzero terms. This enables us to switch the order of the operations (sum, integral, and $\Delta$). Note that the sequence $(\boldsymbol{c}^s[k]^T -\boldsymbol{c}^s[k-1]^T)_{k\inZ}$ is causal.
In short, $\boldsymbol{H}$ can reproduce $(\beta_+^{n_1+1},...,\beta_+^{n_N+1})$ with causal sequences, but it is obviously not a shortest-support basis because its support is infinite.
\subsubsection*{{\it Finite Difference}}
The aim now is to find a basis with the same reproducing properties as $\boldsymbol{H}$, but with minimal support.
To that end, we denote by $s_0$ the index so that $\eta_{s_0}$ is the shortest function in $\boldsymbol{\eta}$ that satisfies $\int_{\mathbbm{R}}\eta_{s_0} \neq 0$. It must exist; if not, the generating $S(\V \eta)$ would only contains zero-mean functions and could not reproduce the B-splines that are not zero-mean. A shortest-support basis $\boldsymbol{\theta} = (\theta_1,...,\theta_N)$ is then given by
\begin{equation}
\theta_s = \begin{cases}
            H_s& \text{if } s\neq s_0 \text{ and } \int_{\mathbbm{R}}\eta_{s}(t)\dint t = 0\\
            H_s-H_{s_0}& \text{if }  s\neq s_0\text{ and } \int_{\mathbbm{R}}\eta_{s} (t)\dint t\neq 0\\
            \Delta H_{s_0}& s = s_0
        \end{cases}
\end{equation}
Because $\boldsymbol{\eta}$ is compactly supported and standardized, the choice of $s_0$ ensures that
\begin{equation}
\suppsize{\theta_s} = \begin{cases}
                    \suppsize{\eta_s}&s\neq s_0\\
                    \suppsize{\eta_{s_0}}+1&s = s_0
                    \end{cases}
\end{equation}
In short, $\suppsize{\boldsymbol{\theta}} = 1 + \suppsize{\boldsymbol{\eta}} = n_N+2$.
Noting that $H_{s_0}=\sum_{k\inN}\theta_{s_0}(\cdot -k)$, it is clear that $\boldsymbol{\theta}$ can reproduce $\boldsymbol{H}$ with causal coefficients. It also implies that $\boldsymbol{\theta}$ can reproduce $(\beta_+^{n_1+1},...,\beta_+^{n_N+1})$ with causal coefficients (see (\ref{reproducecausal})), which justifies the assumption $(A_{\mathbf{n}})$. In conclusion, $\boldsymbol{\theta}$ is a shortest-support basis of  $S_{n_1+1}+\cdots+S_{n_N+1}$.

\subsubsection{Insertion Step}
The present step enables us to add a generator to a shortest-support basis. Suppose $\boldsymbol{\eta} = (\eta_1,...,\eta_{N})$ is a standardized shortest-support basis of $S_{n_1}+\cdots+S_{n_N}$ and let $\boldsymbol{\eta}' = (\delta,\eta_1,...,\eta_{N})$, where $\delta$ is the Dirac distribution. The increment step applied to $\boldsymbol{\eta}' $ yields a shortest-support basis for $S_{0}+S_{n_1+1}+\cdots+S_{n_N+1}$. Indeed, the shortest function of $\boldsymbol{\eta}$ being $\delta$, the new basis $\boldsymbol{\theta}' = (\theta_0',...,\theta_{N}')$ is given by
\begin{equation}
\theta_n':x\mapsto \begin{cases}
           \Delta \{\int_{-\infty}^{x} \delta(t)\dint t\} = \beta_+^0(x), & n=0\\\int_{-\infty}^{x} \eta_n(t)\dint t,& n>0 \text{ and } \int_{\mathbbm{R}}\eta_{n}(t)\dint t = 0\\
             \int_{-\infty}^{x} (\eta_n(t)-\delta(t))\dint t,& n>0\text{ and } \int_{\mathbbm{R}}\eta_{n}(t)\dint t \neq 0 .

        \end{cases}
\end{equation}
Because $\boldsymbol{\eta}$ is compactly supported and standardized, we have that
$$
\suppsize{\theta_n'} = \begin{cases}1,&n=0\\
                    \suppsize{\eta_n},&\text{otherwise,}
                    \end{cases}
$$
which means that $\suppsize{\boldsymbol{\theta}'} = \suppsize{\eta'}+1 = n_N+2$.
The process also ensures that $\boldsymbol{\theta}'$ is a shortest-support basis of $S_{0}+S_{n_1+1}+\cdots+S_{n_N+1}$.

\begin{theorem}
\label{th:existence}
Let $n_1<\cdots<n_N \inN \setminus{\{0\}}$. There exists an mB-spline $\boldsymbol{\eta} = (\eta_1,...,\eta_N)\inLDRN$ of $S_{n_1}+\cdots+S_{n_N}$ that can be constructed recursively with increment and insertion steps.
\end{theorem}
\begin{proof}
The increment and insertion steps are sufficient to construct an mB-spline for any multi-spline space. Indeed, take $\boldsymbol{\eta_0} = (\beta_+^{n_N-n_{N-1}-1})$ a shortest support basis for $S_{n_N-n_{N-1}-1}$. The insertion step gives a shortest-support basis for $S_0+S_{n_N-n_{N-1}}$. After $(n_{N-1}-n_{N-2}-1)$ increment steps and one insertion step, the process gives a shortest-support basis for $S_0+S_{n_{N-1}-n_{N-2}}+S_{n_N-n_{N-2}}$. By iteration, a shortest-support basis for $S_0+S_{n_2-n_1}+\cdots+S_{n_N-n_1}$ is obtained. Applying $n_1$ increment steps, we finally obtain a shortest-support basis for $S_{n_1}+\cdots+S_{n_N}$
\end{proof}

Examples of mB-splines will be provided in Section \ref{sec:Applications}. 
Note that our algorithm does not always output functions with the most practical form. This is corrected by appropriate linear combinations and, possibly, translations that do not alter the reproducing properties and the support size. For instance, for the space $S_2+S_3$, our construction will need a simple linear combination to obtain the wellknown bicubic Hermite splines.
We conclude this section with a result on the minimal number of generating functions required to generate multi-spline spaces.

\begin{theorem}
\label{prop:MN}
Let $n_1<\cdots<n_N\inN\setminus{\{0\}}$. The space $S_{\mathbf{n}} = S_{n_1}+\cdots+S_{n_N}$ cannot be generated by fewer than $N$ compactly supported generating functions.
\end{theorem}
\begin{proof}
From Theorem \ref{th:existence}, there exists an mB-spline of $S_{\mathbf{n}}$ composed of $N$ functions, say, $\V \eta=(\eta_1,\ldots,\eta_N)\in (S_{\mathbf{n}})^N$. Let $\V \psi = (\psi_1,...,\psi_M)\in (S_{\mathbf{n}})^M$ be a collection of compactly supported functions able to generate $S_{\mathbf{n}}$. It means that $\V \eta$ and $\V \psi$ can reproduce each other and, by Lemma \ref{lm:MinimalGeneratingFunctions}, $M\geq N$.
\end{proof}
Note that $N$ is a lower bound and the number of generating function of a shortest-support basis can exceed $N$. For instance, take $\V \eta = (\eta_1,\eta_2)$ with
\begin{align}
\eta_1&:x \mapsto \beta_0(2x) = \mathbbm{1}_{[0,1/2)}(x)\\
\eta_2&:x \mapsto \beta_0(2(x-1/2))  = \mathbbm{1}_{[1/2,1)(x)}
.\end{align}
Since $\eta_1+\eta_2 = \beta_0$, $\V \eta$ can reproduce $S_0$. In addition, the fact that $\suppsize{\V \eta} = 1$ means that it is a shortest-support basis of degree 0 and now it is composed of two generating functions. (Note that the space they generate is larger than $S_{0}$).

\section{Applications}
\label{sec:Applications}
\subsection{Generalized Sampling in Multi-Spline Spaces}
We consider a multi-spline space $S_{\mathbf{n}}$ along with the $N$-component mB-spline $\boldsymbol{\phi}=(\phi_1,\ldots,\phi_N)$ and some corresponding analysis functions $\boldsymbol{\psi} = (\psi_1,\ldots,\psi_N)$. As we now show, the generalized-sampling formulation presented in \cite{Unser1998} can be extended to multiple generators. Let $\mathcal{H}$ be a space considerably larger than $\shiftspace{\boldsymbol{\phi}}$. Consider $f\in \mathcal{H}$, from which we know only some discrete measurements $(\boldsymbol{g}[n])_{n\inZ}$ written
$$
\boldsymbol{g}[n] =\dotprod{\boldsymbol{\psi}(\cdot-n)}{f} = (\dotprod{\psi_1(\cdot-n)}{f},...,\dotprod{\psi_N(\cdot-n)}{f})
.$$
To construct an approximation $\Tilde{f}\in \shiftspace{\boldsymbol{\phi}}$ of $f$, a standard way is to enforce consistency \cite{Unser1994,Unser1997a}, in the sense that $f$ and $\Tilde{f}$ must give the same measurements. This formulation generalizes the notion of interpolation. For instance, to interpolate the value of $f$ and its derivative at the sampling locations, take $\psi_1 = \delta$ and $\psi_2 = \delta^{'}$. In such a case, consistency simply means that $f$ and $\Tilde{f}$ should have the same value and the same derivative at the grid points. In general, the consistency requirement translates into
\begin{align}
\dotprod{\boldsymbol{\psi}(\cdot-n)}{f} &= \dotprod{\boldsymbol{\psi}(\cdot-n)}{\tilde{f}} \nonumber\\
     &= \sum_{k \inZ}\dotprod{\boldsymbol{\psi}(\cdot-n)}{\boldsymbol{\phi}^T(\cdot-k)}\cdot \boldsymbol{c}[k] \nonumber\\
    &= \sum_{k \inZ}\dotprod{\boldsymbol{\psi}(\cdot-(n-k))}{\boldsymbol{\phi}^T}\cdot \boldsymbol{c}[k] \nonumber\\
    & = (\boldsymbol{A}_{\boldsymbol{\Phi \Psi}} * \boldsymbol{c})[n]
\end{align}
where $(\boldsymbol{c}[n])_{n\inZ}$ is the unique vector sequence representing $\Tilde{f} = \sum_{k\inZ}{\V c}[k]^T {\V \phi}(\cdot-k)$ and $\boldsymbol{A}_{\boldsymbol{\Phi \Psi}}[n] = \dotprod{\boldsymbol{\psi}(\cdot-n)}{\boldsymbol{\phi}^T(\cdot)}$ is the matrix-valued sequence of the measurements of the basis functions. To solve our problem, we rely on the theory of signal and systems, including the z-transform. Indeed, with this framework efficient implementation techniques naturally stand out. When the matrix-valued filter $\boldsymbol{A}_{\boldsymbol{\Phi \Psi}}$ is invertible (see  \cite[Proposition 1]{Unser1998} for the invertibility condition), the vector ${\V c}$ of sequences can be computed from the measurements by applying the matrix-valued inverse filter $\V Q$, like in
\begin{equation}
    \V c [n] = (\V Q *\V g)[n]
.\end{equation}
Its transfer function verifies in the z-domain
$
    \hat{\V Q}(z) =  \hat{\V A}_{\V \Phi \V\Psi}^{-1}(z)
$. This matrix filter has not necessarily a finite impulse response (FIR) but it can be decomposed as $\hat{\V Q}(z)= \frac{1}{\det \hat{\V A}_{\V \Phi \V\Psi}(z)}\mathrm{com} (\hat{\V A}_{\V \Phi \V\Psi}(z))^T$, where $\mathrm{com} (\hat{\V A}_{\V \Phi \V\Psi})$ denotes the cofactor matrix of $\hat{\V A}_{\V \Phi \V\Psi}$. For compactly supported analysis functions, the comatrix $\mathrm{com} (\hat{\V A}(z))$ is FIR because it is a Laurent polynomial in $z$, so it is straightforward to implement. On the contrary, $\frac{1}{\det \hat{\bold{A}}_{\V \Phi \V\Psi}(z)}$ is often not FIR. Nonetheless, it can usually be implemented efficiently too, using the same techniques as in \cite{Unser1993a}.

\paragraph{Online Interactive Tutorial}
Some examples are implemented in an online interactive demo  \footnote{\url{https://bigsplinesepfl.github.io/}}, a screenshot being provided in Figure \ref{fig:demoDerivativeSampling}. The user can control the discrete measurements of a function (value, derivative), choose a multi-spline reconstruction space, and see in live the reconstructed function.

\subsection{Derivative Sampling with High-Degree Multi-Splines in $S_{2p}+S_{2p+1}$}
The derivative sampling problem reads for $f\in\mathcal{H}$
\begin{equation}
\label{DerivativeSampling}
\text{find } \tilde{f}\in S_{\mathbf{n}}: \begin{cases}\tilde{f}(k)=f(k)\\\tilde{f}^{'}(k)=f^{'}(k)\end{cases}, k\inZ
.\end{equation}
The most relevant reconstruction spaces have the form $S_{\mathbf{n}} = S_{2p}+S_{2p+1}$. The underlying reason is that the filter complexity is the same for the spaces $S_{2p}+S_{2p+1}$ and $S_{2p-1}+S_{2p}$, so, the higher degree is preferred (the filter has $2(p-1)$ roots). Note that the same occurs when one performs classical interpolation with B-splines and  odd degrees are usually preferred. To the best of our knowledge, when $p>1$, no solution based on shortest-support bases and recursive filtering has been proposed so far. Our construction of shortest-bases results in the functions $\eta_1$ and $\eta_2$. They have a support size $(p+1)$ and are plotted in Figure \ref{fig:derivativeSampling}. Due to the symmetry properties of those functions, the entries of $\boldsymbol{\hat{A}}_{\boldsymbol{\Phi \Psi}}(z)$ have poles that come in reciprocal pairs. Consequently, the inverse matrix filter can be implemented with efficient recursive techniques, as detailed in \cite{Unser1993,Unser1993a}.
\\The case of quintic-degree derivative sampling is detailed now. The basis functions are specified in Table \ref{tab:derivativesampling}.
\begin{figure}[t!]
\begin{minipage}{1.0\linewidth}
  \centering
  \centerline{\includegraphics[width=\linewidth]{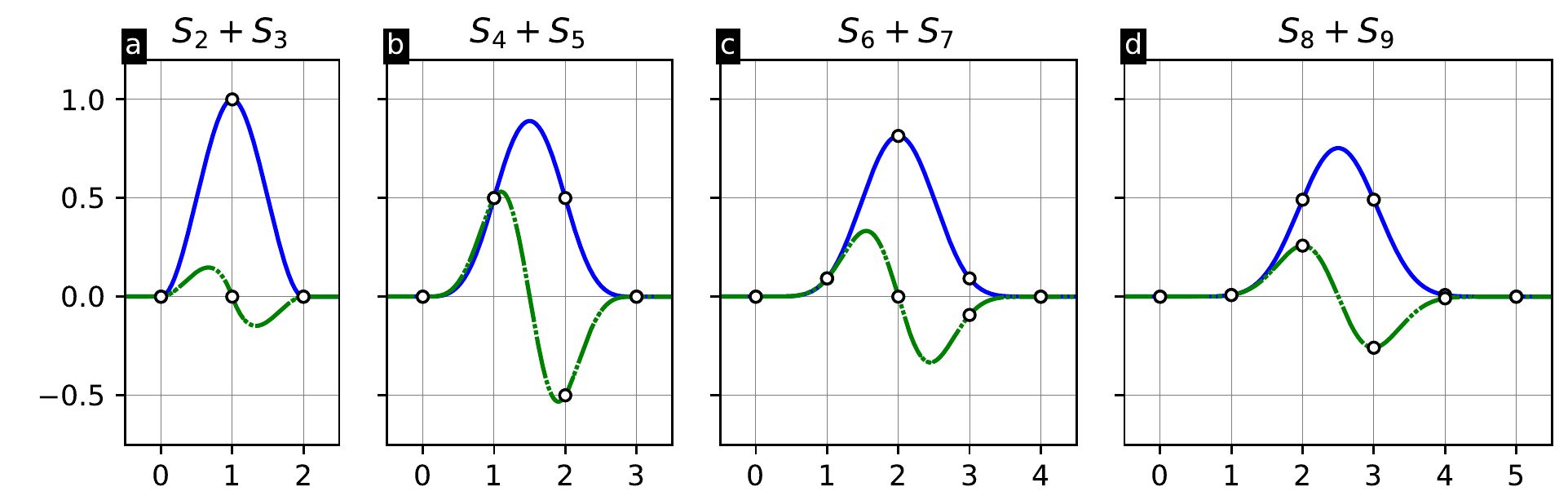}}
  \caption{Shortest-support bases for derivative sampling, obtained with the shortest-basis algorithm and some linear combinations to get a symmetric and an antisymmetric function. (a) The well-known bicubic Hermite splines. (b)-(c)-(d) New bases for derivative sampling with high-degree splines. These functions are piecewise polynomials of degree 5, 7, 9 with continuity of the derivatives of order 3, 5, 7, respectively.}
  \label{fig:derivativeSampling} \medskip
\end{minipage}
\end{figure}
\begin{table}[ht]
\centering
{\renewcommand{\arraystretch}{1}%
\begin{tabular}{l |l|l|r r r r r r r r r r}
\bottomrule\bottomrule
  & & slice $\#$&$x_k^0$&$x_k^1$&$x_k^2$&$x_k^3$&$x_k^4$&$x_k^5$&$x_k^6$&$x_k^7$\\
\hline \multirow{4}{*}{$S_2+S_3$}&\multirow{2}{*}{$\eta_1$} & $k=0$ &&&-3&2&&&&\\
&&$k=1$&1&&-3&1&&&&\\
\cline{2-11}
&\multirow{2}{*}{$\eta_2$} & $k=0$ &&&-1&1&&&&\\
&&$k=1$&&1&-2&1&&&&\\
\hline\multirow{6}{*}{$S_4+S_5$}&\multirow{3}{*}{4$\eta_1$} & $k=0$ &&&&&5&-3&&\\
 & &$k=1$&2&5&&-10&5&&&\\
 & &$k=2$&2&-5&&10&-10&3&&\\
 \cline{2-11}&  \multirow{3}{*}{8$\eta_2$} & $k=0$ &&&&&15&-11&&\\
 & &$k=1$&4&5&-20&-50&95&-38&&\\
 & &$k=2$&-4&5&20&-50&40&-11&&\\
 \hline\multirow{8}{*}{$S_6+S_7$}&\multirow{4}{*}{$108\eta_1$}&$k=0$&&&&&&&21&-11\\
&&$k=1$&10&49&84&35&-70&-105&112&-27\\
&&$k=2$&88&&-168&&140&&-77&27\\
&&$k=3$&10&-49&84&-35&-70&105&-56&11\\
 \cline{2-11}&\multirow{4}{*}{$\frac{918}{5}\eta_2$}&$k=0$&&&&&&&42&-25\\
&&$k=1$&17&77&105&-35&-245&-273&539&-185\\
&&$k=2$&&-224&&560&&-924&756&-185\\
&&$k=3$&-17&77&-105&-35&245&-273&133&-25\\ \toprule \toprule
\end{tabular}
}
\caption{Slices of shortest-support bases for derivative sampling. The slices are given as linear combinations of the shifted monomials $x_k^n = (x-k)^n$ if $x\in [k,k+1)$ and $x_k^n = 0$ otherwise.}
\label{tab:derivativesampling}
\end{table}
The z-transform of the filter $\boldsymbol{\hat{A}}_{\boldsymbol{\Phi \Psi}}(z)$ reads
\begin{equation}
\boldsymbol{\hat{A}}_{\boldsymbol{\Phi \Psi}}(z) = \begin{bmatrix}
\frac{z^{-1}+z^{-2}}{2}&\frac{z^{-1}-z^{-2}}{2}\\
\frac{5(z^{-1}-z^{-2})}{4}&\frac{5(z^{-1}+z^{-2})}{8}
\end{bmatrix}
.\end{equation}
It follows that the transpose comatrix satisfies
\begin{equation}
\rm com (\hat{\bold{A}}(z))^T \quad \xleftrightarrow{ \quad z \quad} \quad \frac{1}{2} \begin{bmatrix}
\frac{5(\delta[\cdot-1]+\delta[\cdot-2])}{4}&-\delta[\cdot-1]+\delta[\cdot-2]\\
-\frac{5(\delta[\cdot-1]-\delta[\cdot-2])}{2}&{\delta[\cdot-1]+\delta[\cdot-2]}
\end{bmatrix} 
\end{equation}
and the determinant
\begin{equation}
\frac{z^{-1}}{\det \hat{\bold{A}}(z)}=\frac{16}{5}\frac{-z}{(1-z_0z^{-1})(1-z_0^{-1}z^{-1})} \quad \xleftrightarrow{ \quad z \quad} \quad d[n],
\end{equation}
where $z_0 = (3-2\sqrt{2})$. This means that the convolution of any sequence with $d$ can be implemented recursively. Interestingly, it is the same inverse filter as in cubic-spline interpolation. The reader can therefore refer to \cite{Unser1999} for a detailed explanation of the implementation.
The expansion coefficients can be evaluated as
\begin{align}
c_1&= d*\left (\frac{5}{8}\Delta^+\{f\}-\frac{1}{2}\Delta \{f^{'}\}\right) \nonumber \\
c_2 &= 
d*\left(-\frac{5}{4}\Delta \{f\} + \frac{1}{2}\Delta^+ \{f^{'}\}\right)
,\end{align}
where  $\Delta^+ \{f\} [k] = f[k]+f[k-1]$.  Finally, the multi-spline that is consistent with the measurements is given by
\begin{equation}
\tilde{f}(x) = \sum_{k\inZ}c_1[k]\eta_1(x-k)+\sum_{k\inZ}c_2[k]\eta_2(x-k)
.
\end{equation}
\begin{figure}[t!]
\begin{minipage}{1.0\linewidth}
  \centering
  \centerline{\includegraphics[width=\linewidth]{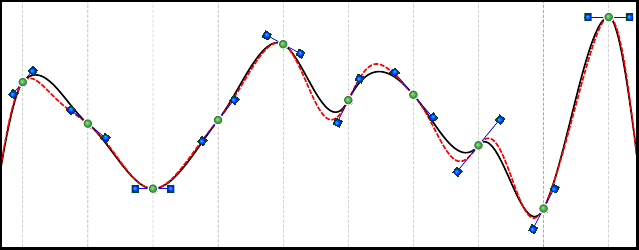}}
  \caption{Derivative sampling with optimal bases. The solid curve lies in $S_2+S_3$ (cubic piecewise polynomials with continuous derivative) and the dashed curve lies in $S_4+S_5$ (quintic piecewise polynomials with continuous third derivative).}
  \label{fig:demoDerivativeSampling} \medskip
\end{minipage}
\end{figure}

\subsubsection{Derivative Sampling in $S_2+S_3+S_{4}$}
Here, we consider the setting $\V \psi = (\delta,\delta^{'},\delta(\cdot-1/2))$, which means that the value of the function to be reconstructed is sampled twice more often than its derivative. The specification of $S_2+S_3+S_4$ as reconstruction space provides then an explicit interpolation formula, which involves the shortest-support basis $\V \eta$, plotted in Figure \ref{fig:samplingS2S3S4}. This formula reads
\begin{equation}
\tilde{f}(x) = \sum_{k\inZ}\left ( f(k)\eta_1(x-k)+f'(k)\eta_2(x-k) +f(k+1/2)\eta_3(x-k)\right ).
\end{equation}
\begin{figure}[t!]
\begin{minipage}{1.0\linewidth}
  \centering
  \centerline{\includegraphics[width=\linewidth]{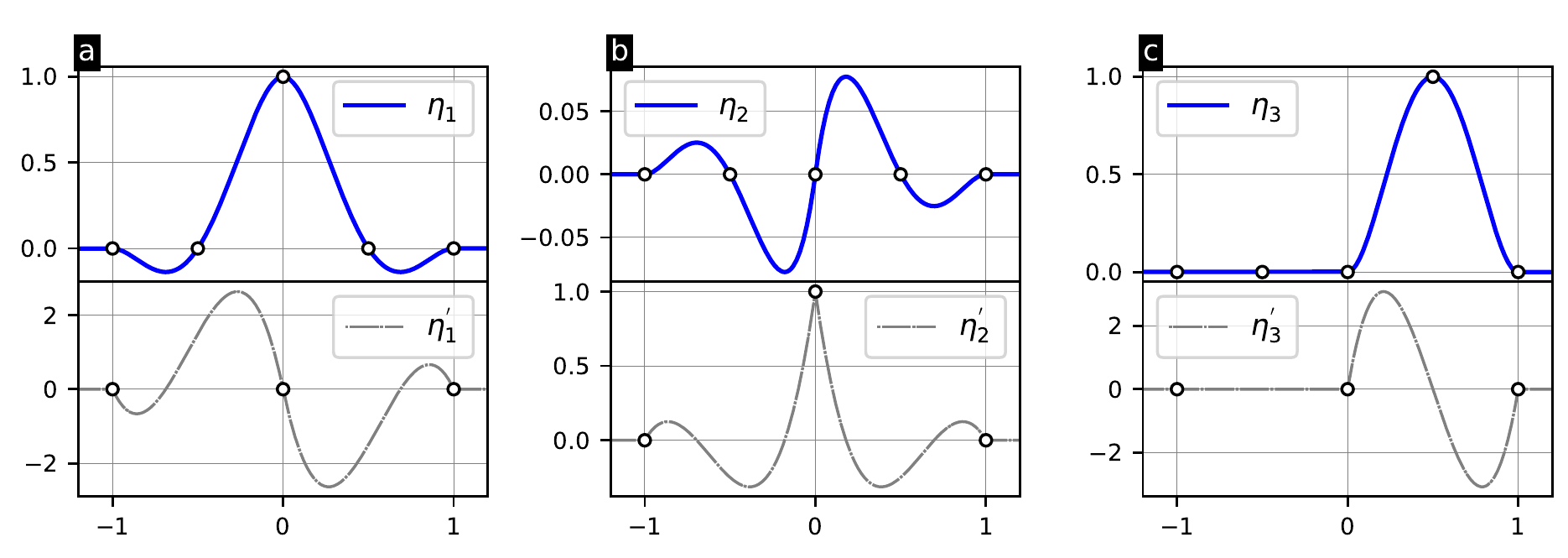}}
  \caption{Shortest basis of $S_2+S_3+S_4$ associated to the analysis
  functions $\V \psi = (\delta,\delta^{'},\delta(\cdot-1/2))$.}
  \label{fig:samplingS2S3S4} \medskip
\end{minipage}
\end{figure}
More generally, we observed that the addition of $N$ consecutive spline spaces to $S_2+S_3$ ({\it i.e.}, choosing $S_{2}+S_{3}+\cdots+S_{3+N}$) allows one to perform derivative sampling and interpolate the function $N$ times between the integers with a direct interpolation formula.
\subsubsection{Direct Derivative Sampling in $S_2+\cdots+S_{2p+1}$}
The space $S_2+S_3+S_4+S_5$ is also well suited for derivative sampling with $\psi = (\delta,\delta',\delta(\cdot-1/2),\delta'(\cdot-1/2))$ because of the structure of its shortest-support generating functions $\eta_1,\eta_2,\eta_3$, and $\eta_4$ (Figure \ref{fig:doubleDerivativeSampling}). Indeed, it yields the direct interpolation formula
\begin{equation}
\tilde{f}(x) = \sum_{k\inZ} \left (f(k+1/2)\eta_1(x-k)+f(k)\eta_2(x-k+1) +f'(k+1/2)\eta_3(x-k)+f'(k)\eta_4(x-k+1) \right )
.
\end{equation}
The sampling step is $1/2$, but the spline knots are still located at the integers. Note that the sampling step can be tuned at will by dilation of the generating functions. More generally, we conjecture that there exist basis functions with the interpolatory property for any space of the form $S_2+\cdots+S_{2p+1}$ and the sampling step $1/p$. This conjecture was verified for $p=1$ (bicubic Hermite splines), $p=2$ (Figure \ref{fig:doubleDerivativeSampling}) and $p\in\{3,4\}$.
\begin{figure}[t!]
\begin{minipage}{1.0\linewidth}
  \centering
  \centerline{\includegraphics[width=\linewidth]{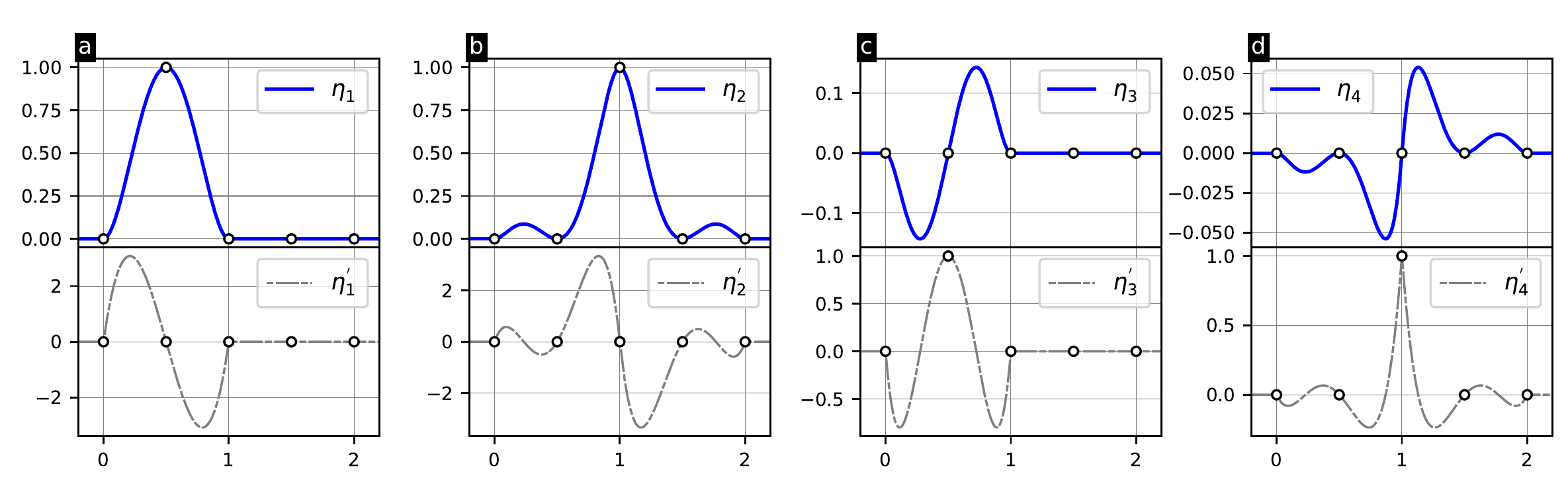}}
  \caption{Shortest basis of $S_2+S_3+S_4+S_5$ for direct derivative sampling.}
  \label{fig:doubleDerivativeSampling} \medskip
\end{minipage}
\end{figure}

\subsection{Classical Interpolation}
The classical interpolation problem reads for $f\in\mathcal{H}$
\begin{equation}
\label{ClassicalInterpolation}
\text{find } \tilde{f}\in S_{\mathbf{n}}: f(k)=\tilde{f}(k), k\inZ
.\end{equation}
When the number $N$ of generating functions is greater than 1, we have two equivalent options:
\begin{enumerate}[label=(\roman*)]
\item to sample the function $f$ with the sampling step $1/N$;
\item to dilate the generators by a factor of $N$, keeping a unit sampling step.
\end{enumerate}
We present the result in accordance with Option (i).

\subsubsection{Modified Lagrange Polynomials in $S_1+\cdots+S_N$}
Classical interpolation is well solved by B-splines but, starting from degree 2, the filter is neither FIR nor causal. Exact operations such as local interpolation or interpolation with a finite delay are therefore not possible. Some workarounds exist \cite{Petrinovic2008}; we present now one that is based on modified Lagrange polynomials. Let $\V l= (l_1,\ldots,l_N)$ be a collection of $N$ generating function such that, for $x\in[0,1]$,
$l_q(x)=\prod_{\substack{p = 0\\p\neq q}}^{N}\frac{Nx-p}{q-p}$. In this way, when $q=1,\ldots,(N-1)$, $l_q$ is zero at $x=0$ and $x=1$ so it can be set to zero for $x\not\in[0,1]$ and $l_q\in S_1$. Noting that $l_N(1)=1$, to make sure that $l_N\in S_1$, we extend its support to $[1,2]$ and set, $\forall x\in [1,2]$, $l_N(x)=l_N(2-x)$ (see Figure \ref{fig:modifiedLagrange}).
These functions constitute a shortest-support basis of $S_1+\cdots+S_N$ and give a direct interpolation formula. Interestingly, those basis functions are sometimes used for finite-element methods \cite{Langtangen2019}.
\begin{figure}[t!]
\begin{minipage}{1.0\linewidth}
  \centering
  \centerline{\includegraphics[width=\linewidth]{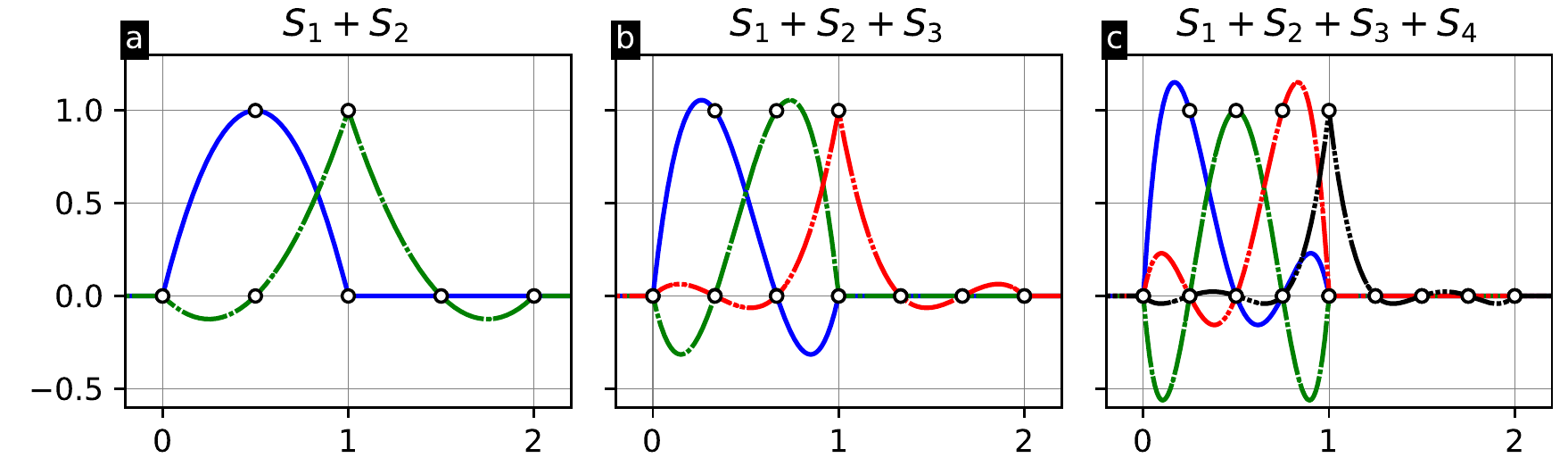}}
  \caption{Shortest-support basis for $S_1+\cdots+S_N$. The basis functions are continuous and able to reproduce any polynomial of degree up to $N$.}
  \label{fig:modifiedLagrange} \medskip
\end{minipage}
\end{figure}

\subsubsection{Bi-Spline Classical Interpolation in $S_{2p+1}+S_{2p+2}$}
A bi-spline is the sum of two splines of different degrees, and it can be used to perform classical interpolation. In particular, interpolation in the reconstruction space $S_{\mathbf{n}} = S_{2p+1}+S_{2p+2}$ leads to a filter with $p$ pairs of reciprocal roots. In terms of filtering, it has therefore the same complexity as for the interpolation inverse filter associated with the single space $S_{2p+1}$. Shortest-support  basis functions for such spaces are plotted in Figure \ref{fig:bisplineInterpolation}.
\begin{figure}[t!]
\begin{minipage}{1.0\linewidth}
  \centering
  \centerline{\includegraphics[width=\linewidth]{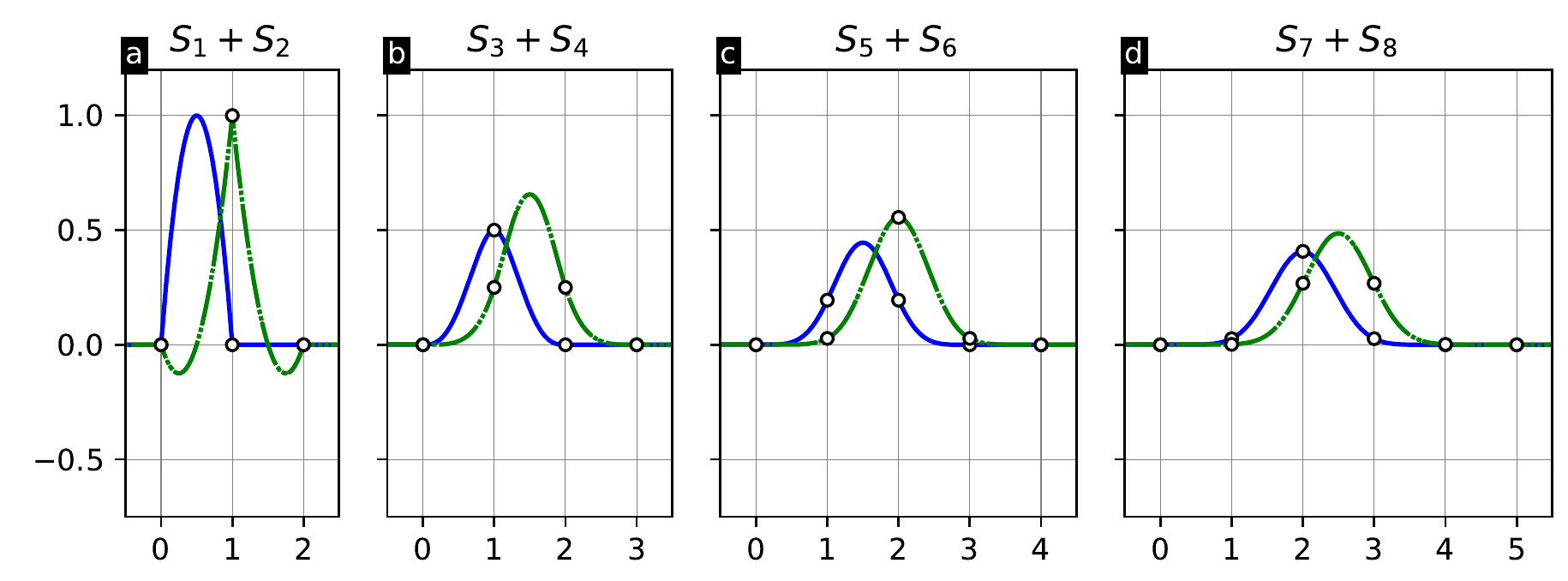}}
  \caption{Shortest bi-spline bases for classical interpolation with a half-integer sampling step. (a) In $S_1+S_2$, the functions presented give a direct interpolation formula. (b) (c) (d) The functions are piecewise polynomials of degree 4, 6, 8 with continuity of the derivatives of order 2, 4, 6 respectively. To perform interpolation, a filter with 2, 4, 6 roots respectively has to be inverted.}
  \label{fig:bisplineInterpolation} \medskip
\end{minipage}
\end{figure}
We now detail how this interpolation is performed for $S_3+S_4$, keeping in mind that the other cases are similar.
The z-transform of the filter $\boldsymbol{\hat{A}}_{\boldsymbol{\Phi \Psi}}(z)$ reads
\begin{equation}
\boldsymbol{\hat{A}}_{\boldsymbol{\Phi \Psi}}(z) = \begin{bmatrix}
\frac{z^{-1}}{2}&\frac{z^{-1}+z^{-2}}{4}\\
\frac{5(z^{-1}+z^{-2})}{32}&\frac{5(z^{-1}+z^{-3})+210z^{-2}}{320}
\end{bmatrix}
,\end{equation}
while the z-transform of the inverse filter can be decomposed as
\begin{equation}
\hat{\boldsymbol{Q}}(z)  = \hat{p}(z)\times \hat{\boldsymbol{P}}(z),
\end{equation}
where
\begin{equation}
\hat{\boldsymbol{P}}(z) = \begin{bmatrix}
\frac{5(1+z^{-2})+210z^{-1}}{320}&-\frac{1+z^{-1}}{4}\\
-\frac{5(1+z^{-1})}{32}&\frac{1}{2}
\end{bmatrix} 
\end{equation}
and
\begin{equation}
\hat{p}(z)=\frac{32}{(1-z_0z^{-1})(1-z_0^{-1}z^{-1})}
\end{equation}
with $z_0 = (4-\sqrt{15})$.  The final steps are identical to the detailed case of derivative sampling (recursive filtering).
\subsection{B\'ezier Curves  and Computer Graphics in $S_1+S_2+S_3$ and $S_1+S_2$}
In this section, we use our multi-spline formulation to revisit some B\'ezier curves and, in particular, the cubic B\'ezier curves that are popular in computer graphics. Each portion of the curve is a cubic polynomial defined by four control points.
\begin{itemize}
\item Starting point and ending point of the portion.
\item Two handles that control the tangent of the curve at each extremity of the portion.
\end{itemize}
Thus, the value of the function and its left and right derivatives are controlled on the knots. From a multi-spline perspective, any cubic B\'ezier curve lies in the space $S_1+S_2+S_3$. With the well chosen generating functions $\eta_1,\eta_2$, and $\eta_3$ plotted in Figure \ref{fig:bezierMS}, the interpolation formula is explicit and reads
\begin{equation}
\tilde{f}(x) = \sum_{k\inZ}f(k)\eta_1(x-k)+\sum_{k\inZ}f^{'}(k^-)\eta_2(x-k)+\sum_{k\inZ}f^{'}(k^+)\eta_3(x-k)
,
\end{equation}
where $f'(k^-)$ and $f'(k^{+})$ denote the left and right derivatives at $k$, respectively.
Interestingly, $\eta_2$ and $\eta_3$ can be obtained from the bi-cubic Hermite splines, by splitting the antisymmetric function into two functions (see Figure \ref{fig:derivativeSampling} (a)).
It gives a simple interpretation to cubic B\'ezier curves as illustrated in Figure \ref{fig:2DBezier}. Similarly, quadratic B\'ezier curves are also multi-splines, this time associated to the space $S_1+S_2$ (Figure \ref{fig:bezierMS}). 
\begin{figure}[t!]
\begin{minipage}{1.0\linewidth}
  \centering
  \centerline{\includegraphics[width=\linewidth]{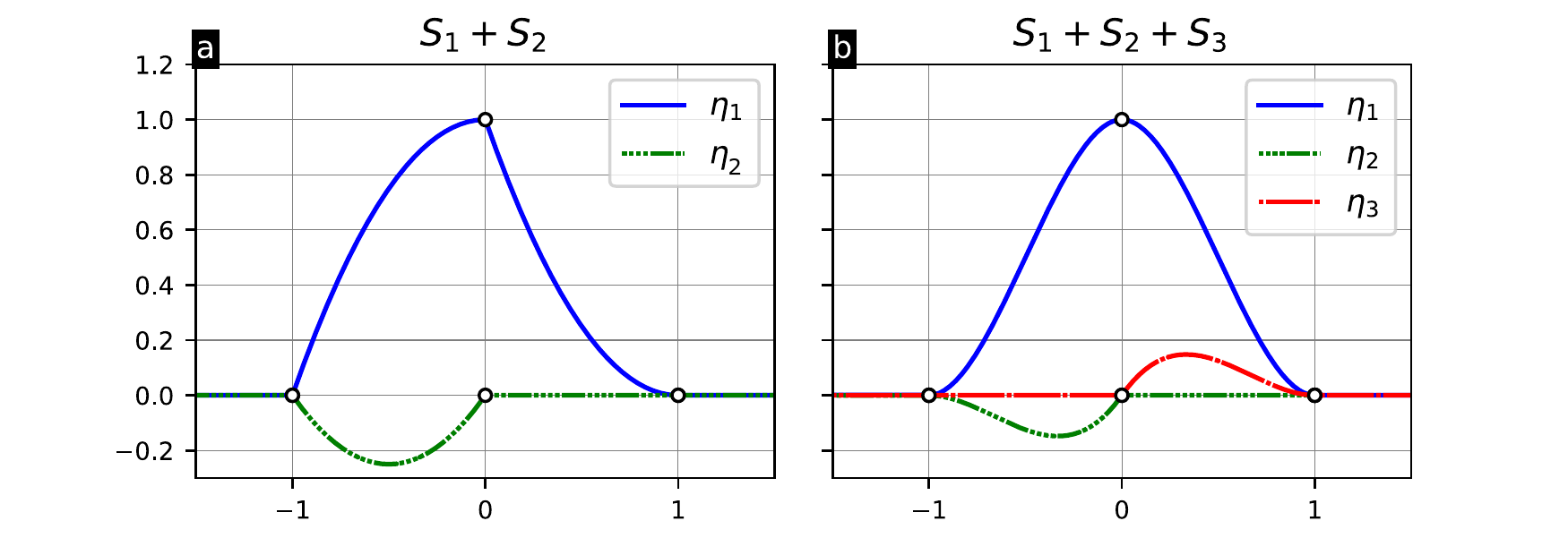}}
  \caption{Shortest-support bases for application in classical computer-graphics. (a) Shortest basis for $S_1+S_2$. The function $\eta_1$ controls the value of the function on the knots while $\eta_2$ controls the left derivative on the knots. These functions reproduce any quadratic B\'ezier curve. (b) Shortest basis for $S_1+S_2+S_3$. The function $\eta_1$ controls the value of the function on the knots while $\eta_2$ and $\eta_3$ control the left and right derivatives, respectively, on the knots. These functions can reproduce any cubic B\'ezier curve with the shortest support. They also give a simple interpretation of such curves.}
  \label{fig:bezierMS} \medskip
\end{minipage}
\end{figure}

\begin{figure}[t!]
\begin{minipage}{1.0\linewidth}
  \centering
  \centerline{\includegraphics[width=100mm]{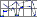}}
  \caption{Screenshot from the online demo. The shortest basis of the space $S_1+S_2+S_3$ allows one to control the value of the function (green dots) and the left/right derivatives (handles). It yields the same curve as with standard vector-graphics editors relying on cubic B\'ezier curves. In this figure, the parametric curves are two-dimensional and the interpolation is performed component-wise.}
  \label{fig:2DBezier} \medskip
\end{minipage}
\end{figure}

\subsection{Nonconsecutive Bi-spline Spaces}
Nonconsecutive multi-spline spaces are relevant to represent signals that have components of different regularity \cite{debarre2019}. For instance, the space $S_0+S_p$, with $p>0$, consists of smooth signals with sharp jumps. In Figure \ref{fig:hybridBiSplines}, we show shortest-support bases of $S_0+S_p$, for $p\in\{2,3,4\}$, that were obtained with our construction algorithm.
\begin{figure}[t!]
\begin{minipage}{1.0\linewidth}
  \centering
  \centerline{\includegraphics[width=0.8\linewidth]{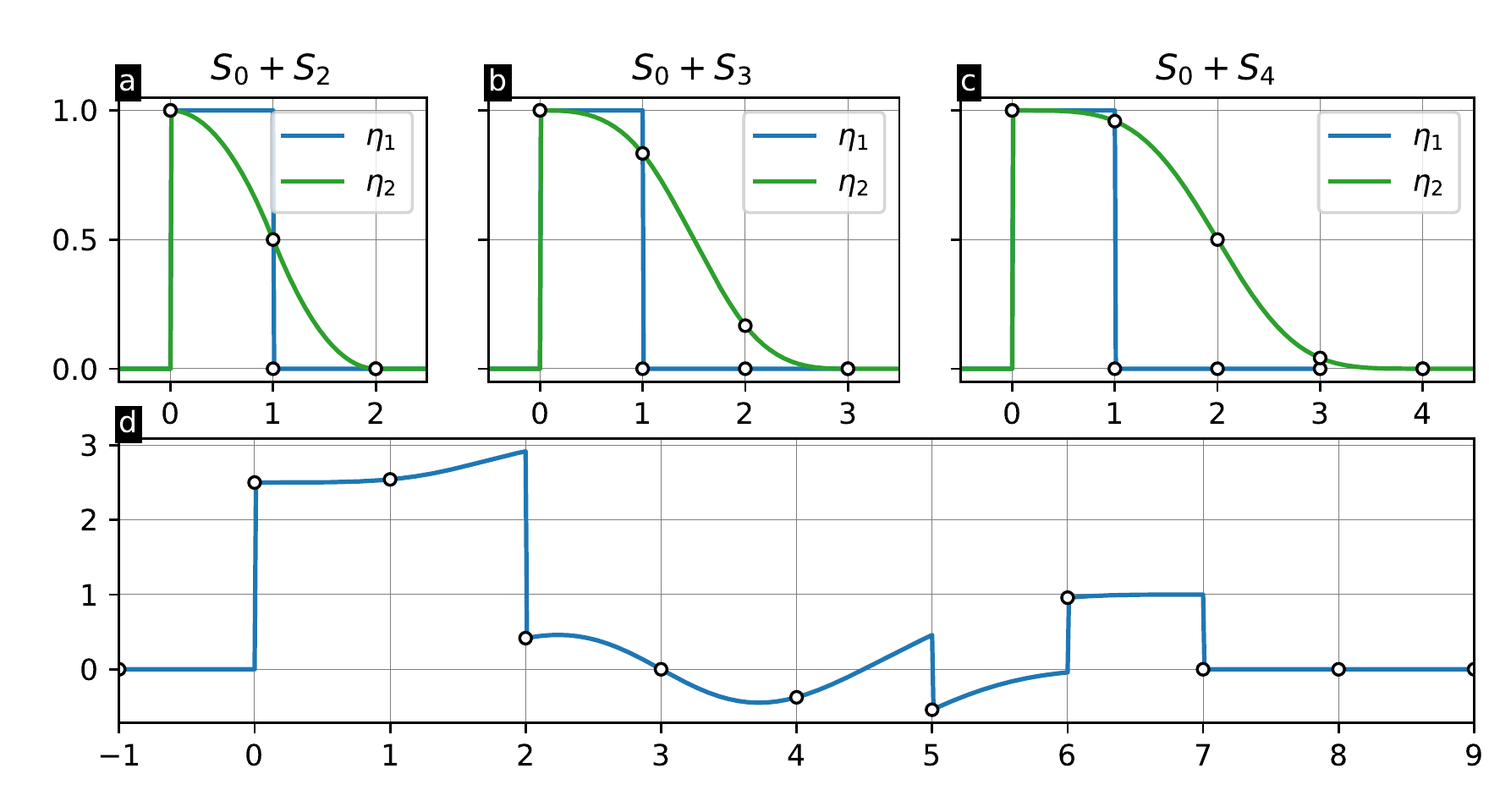}}
  \caption{(a) (b) (c) Shortest-support bases for the spaces $S_0+S_2$, $S_0+S_3$ and $S_0+S_4$. (d) An example of a hybrid bi-spline that lies in the space $S_0+S_4$.}
  \label{fig:hybridBiSplines} \medskip
\end{minipage}
\end{figure}
\section{Conclusion}
In this work, we have introduced the notion of shortest-support bases of degree $M$. They are the shortest-support collections of functions that generate a reconstruction space with an approximation power of order $(M+1)$. We proved that shortest-support bases necessarily generate Riesz bases, a minimal requirement for practical applications. With a single generator, the unique shortest-support basis of degree $M$ is the well-known B-spline of degree $M$. We extended this notion to multiple generators and proposed a recursive method that yields shortest bases for any multi-spline space. These new sets of functions helped us transpose the efficient reconstruction techniques developed for B-splines, and perform generalized sampling. In particular, we have provided a method to perform fast derivative sampling with any approximation power. Finally, we presented a new way to approach some B\'ezier curves.
\newpage
\bibliography{References}

\begin{thebibliography}{10}
\expandafter\ifx\csname url\endcsname\relax
  \def\url#1{\texttt{#1}}\fi
\expandafter\ifx\csname urlprefix\endcsname\relax\def\urlprefix{URL }\fi
\expandafter\ifx\csname href\endcsname\relax
  \def\href#1#2{#2} \def\path#1{#1}\fi

\bibitem{Shannon1949}
C.~E. Shannon, {Communication in the Presence of Noise}, Proceedings of the IRE
  37~(1) (1949) 10--21.
\newblock \href {https://doi.org/10.1109/JRPROC.1949.232969}
  {\path{doi:10.1109/JRPROC.1949.232969}}.

\bibitem{Jerri1977}
A.~J. Jerri, {The Shannon Sampling Theorem–Its Various Extensions and
  Applications: A Tutorial Review}, Proceedings of the IEEE 65~(11) (1977)
  1565--1596.
\newblock \href {https://doi.org/10.1109/PROC.1977.10771}
  {\path{doi:10.1109/PROC.1977.10771}}.

\bibitem{Unser2000}
M.~Unser, {Sampling–50 Years After Shannon}, Proceedings of the IEEE 88~(4)
  (2000) 569--587.
\newblock \href {https://doi.org/10.1109/5.843002}
  {\path{doi:10.1109/5.843002}}.

\bibitem{Papoulis1977}
A.~Papoulis, {Generalized Sampling Expansion}, IEEE Transactions on Circuits
  and Systems 24~(11) (1977) 652--654.
\newblock \href {https://doi.org/10.1109/TCS.1977.1084284}
  {\path{doi:10.1109/TCS.1977.1084284}}.

\bibitem{Aldroubi1994}
A.~Aldroubi, M.~Unser, {Sampling Procedures in Function Spaces and Asymptotic
  Equivalence with Shannon's Sampling Theory}, Numerical Functional Analysis
  and Optimization 15~(1-2) (1994) 1--21.
\newblock \href {https://doi.org/10.1080/01630569408816545}
  {\path{doi:10.1080/01630569408816545}}.

\bibitem{Unser1994}
M.~Unser, A.~Aldroubi, {A General Sampling Theory for Nonideal Acquisition
  Devices}, IEEE Transactions on Signal Processing 42~(11) (1994) 2915--2925.
\newblock \href {https://doi.org/10.1109/78.330352}
  {\path{doi:10.1109/78.330352}}.

\bibitem{Hummel1983}
R.~Hummel, {Sampling for Spline Reconstruction}, SIAM Journal on Applied
  Mathematics 43~(2) (1983) 278--288.
\newblock \href {https://doi.org/10.1137/0143019} {\path{doi:10.1137/0143019}}.

\bibitem{Unser1992}
M.~Unser, A.~Aldroubi, M.~Eden, {Polynomial Spline Signal Approximations:
  Filter Design and Asymptotic Equivalence with Shannon's Sampling Theorem},
  IEEE Transactions on Information Theory 18~(1) (1992) 95--103.
\newblock \href {https://doi.org/10.1109/18.108253}
  {\path{doi:10.1109/18.108253}}.

\bibitem{Aldroubi1992}
A.~Aldroubi, M.~Unser, M.~Eden, {Cardinal Spline Filters: Stability and
  Convergence to the Ideal Sinc Interpolator}, Signal Processing 28~(2) (1992)
  127--138.
\newblock \href {https://doi.org/10.1016/0165-1684(92)90030-Z}
  {\path{doi:10.1016/0165-1684(92)90030-Z}}.

\bibitem{Unser1998}
M.~Unser, J.~Zerubia, {A Generalized Sampling Theory without Band-Limiting
  Constraints}, IEEE Transactions on Circuits and Systems II: Analog and
  Digital Signal Processing 45~(8) (1998) 959--969.
\newblock \href {https://doi.org/10.1109/82.718806}
  {\path{doi:10.1109/82.718806}}.

\bibitem{Unser1997a}
M.~Unser, J.~Zerubia, {Generalized Sampling: Stability and Performance
  Analysis}, IEEE Transactions on Signal Processing 45~(12) (1997) 2941--2950.
\newblock \href {https://doi.org/10.1109/78.650255}
  {\path{doi:10.1109/78.650255}}.

\bibitem{Garcia2008}
A.~G. Garc{\'{i}}a, M.~A. Hern{\'{a}}ndez-Medina,
  G.~P{\'{e}}rez-Villal{\'{o}}n, {Generalized Sampling in Shift-Invariant
  Spaces with Multiple Stable Generators}, Journal of Mathematical Analysis and
  Applications 337~(1) (2008) 69--84.
\newblock \href {https://doi.org/10.1016/j.jmaa.2007.03.083}
  {\path{doi:10.1016/j.jmaa.2007.03.083}}.

\bibitem{Pohl2012}
V.~Pohl, H.~Boche, {U-Invariant Sampling and Reconstruction in Atomic Spaces
  with Multiple Generators}, IEEE Transactions on Signal Processing 60~(7)
  (2012) 3506--3519.
\newblock \href {https://doi.org/10.1109/TSP.2012.2193576}
  {\path{doi:10.1109/TSP.2012.2193576}}.

\bibitem{Radha2019}
R.~Radha, K.~Sarvesh, S.~Sivananthan, {Sampling and Reconstruction in a Shift
  Invariant Space with Multiple Generators}, Numerical Functional Analysis and
  Optimization 40~(4) (2019) 365--385.
\newblock \href {https://doi.org/10.1080/01630563.2018.1501701}
  {\path{doi:10.1080/01630563.2018.1501701}}.

\bibitem{DeBoor1994a}
C.~de~Boor, R.~A. DeVore, A.~Ron, {The Structure of Finitely Generated
  Shift-Invariant Spaces in L2(Rd)}, Journal of Functional Analysis 119~(1)
  (1994) 37--78.
\newblock \href {https://doi.org/10.1006/jfan.1994.1003}
  {\path{doi:10.1006/jfan.1994.1003}}.

\bibitem{Aldroubi1996}
A.~Aldroubi, {Oblique Projections in Atomic Spaces}, Proceedings of the
  American Mathematical Society 124~(7) (1996) 2051--2060.
\newblock \href {https://doi.org/10.1090/S0002-9939-96-03255-8}
  {\path{doi:10.1090/S0002-9939-96-03255-8}}.

\bibitem{Grochenig2018}
K.~Gr{\"{o}}chenig, J.~L. Romero, J.~St{\"{o}}ckler, {Sampling Theorems for
  Shift-Invariant Spaces, Gabor Frames, and Totally Positive Functions},
  Inventiones Mathematicae 211~(3) (2018) 1119--1148.
\newblock \href {https://doi.org/10.1007/s00222-017-0760-2}
  {\path{doi:10.1007/s00222-017-0760-2}}.

\bibitem{DeBoor1994}
C.~de~Boor, R.~A. DeVore, A.~Ron, {Approximation from Shift-Invariant Subspaces
  of L2 (Rd)}, Transactions of the American Mathematical Society 341~(2) (1994)
  787--806.
\newblock \href {https://doi.org/10.2307/2154583} {\path{doi:10.2307/2154583}}.

\bibitem{DeBoor1985}
C.~de~Boor, R.~A. DeVore, {Partitions of Unity and Approximation}, Proceedings
  of the American Mathematical Society 93~(4) (1985) 705--709.
\newblock \href {https://doi.org/10.1090/s0002-9939-1985-0776207-2}
  {\path{doi:10.1090/s0002-9939-1985-0776207-2}}.

\bibitem{Blu2001}
T.~Blu, P.~Th{\'{e}}venaz, M.~Unser, {MOMS: Maximal-Order Interpolation of
  Minimal Support}, IEEE Transactions on Image Processing 17~(7) (2001)
  1069--1080.
\newblock \href {https://doi.org/10.1109/83.931101}
  {\path{doi:10.1109/83.931101}}.

\bibitem{Unser1999}
M.~Unser, {Splines: A Perfect Fit for Signal and Image Processing}, IEEE Signal
  Processing Magazine 16~(6) (1999) 22--38.
\newblock \href {https://doi.org/10.1109/79.799930}
  {\path{doi:10.1109/79.799930}}.

\bibitem{Schoenberg1973}
I.~J. Schoenberg, {Cardinal Spline Interpolation}, SIAM, 1973.
\newblock \href {https://doi.org/10.1137/1.9781611970555}
  {\path{doi:10.1137/1.9781611970555}}.

\bibitem{schoenberg1967spline}
I.~J. Schoenberg, {On Spline Interpolation at all Integer Points of the Real
  Axis}, S{\'{e}}minaire Delange-Pisot-Poitou. Th{\'{e}}orie des nombres 9~(1)
  (1967) 1--18.

\bibitem{DeBoor1976}
C.~de~Boor, {Splines as Linear Combinations of B-splines. A Survey},
  Approximation Theory (1976).
\newblock \href {https://doi.org/10.1.1.34.8204} {\path{doi:10.1.1.34.8204}}.

\bibitem{DeBoor1972}
C.~de~Boor, {On Calculating with B-Splines}, Journal of Approximation Theory
  6~(1) (1972) 50--62.
\newblock \href {https://doi.org/10.1016/0021-9045(72)90080-9}
  {\path{doi:10.1016/0021-9045(72)90080-9}}.

\bibitem{DeBoor1980}
C.~de~Boor, {A Practical Guide to Splines}, Springer-Verlag New York, 1978.
\newblock \href {https://doi.org/10.2307/2006241} {\path{doi:10.2307/2006241}}.

\bibitem{Unser1993}
M.~Unser, A.~Aldroubi, {B-Spline Signal Processing: Part I–Theory}, IEEE
  Transactions on Signal Processing 41~(2) (1993) 821--833.
\newblock \href {https://doi.org/10.1109/78.193220}
  {\path{doi:10.1109/78.193220}}.

\bibitem{Unser1993a}
M.~Unser, A.~Aldroubi, M.~Eden, {B-Spline Signal Processing: Part
  II–Efficient Design and Applications}, IEEE Transactions on Signal
  Processing 41~(2) (1993) 834--848.
\newblock \href {https://doi.org/10.1109/78.193221}
  {\path{doi:10.1109/78.193221}}.

\bibitem{Lipow1973}
P.~R. Lipow, I.~J. Schoenberg, {Cardinal Interpolation and Spline Functions.
  III. Cardinal Hermite Interpolation}, Linear Algebra and Its Applications 6
  (1973) 273--304.
\newblock \href {https://doi.org/10.1016/0024-3795(73)90029-3}
  {\path{doi:10.1016/0024-3795(73)90029-3}}.

\bibitem{Fageot2020}
J.~Fageot, S.~Aziznejad, M.~Unser, V.~Uhlmann, {Support and Approximation
  Properties of Hermite Splines}, Journal of Computational and Applied
  Mathematics 368~(112503) (2020) 1--15.
\newblock \href {https://doi.org/10.1016/j.cam.2019.112503}
  {\path{doi:10.1016/j.cam.2019.112503}}.

\bibitem{Farouki2012}
R.~T. Farouki, {The Bernstein Polynomial Basis: A Centennial Retrospective},
  Computer Aided Geometric Design 26~(6) (2012) 379--419.
\newblock \href {https://doi.org/10.1016/j.cagd.2012.03.001}
  {\path{doi:10.1016/j.cagd.2012.03.001}}.

\bibitem{Uhlmann2016}
V.~Uhlmann, J.~Fageot, M.~Unser, {Hermite Snakes with Control of Tangents},
  IEEE Transactions on Image Processing 25~(6) (2016) 2803--2816.
\newblock \href {https://doi.org/10.1109/TIP.2016.2551363}
  {\path{doi:10.1109/TIP.2016.2551363}}.

\bibitem{Conti2015}
C.~Conti, L.~Romani, M.~Unser, {Ellipse-Preserving Hermite Interpolation and
  Subdivision}, Journal of Mathematical Analysis and Applications 426~(1)
  (2015) 221--227.
\newblock \href {https://doi.org/10.1016/j.jmaa.2015.01.017}
  {\path{doi:10.1016/j.jmaa.2015.01.017}}.

\bibitem{Conti2016}
C.~Conti, M.~Cotronei, T.~Sauer, {Factorization of Hermite subdivision
  operators preserving exponentials and polynomials}, Advances in Computational
  Mathematics 42~(5) (2016) 1055--1079.
\newblock \href {https://doi.org/10.1007/s10444-016-9453-4}
  {\path{doi:10.1007/s10444-016-9453-4}}.

\bibitem{Romani2020}
L.~Romani, A.~Viscardi, {On the Refinement Matrix Mask of Interpolating Hermite
  Splines}, Applied Mathematics Letters 109 (2020) 106524.
\newblock \href {https://doi.org/10.1016/j.aml.2020.106524}
  {\path{doi:10.1016/j.aml.2020.106524}}.

\bibitem{Christensen2016}
O.~Christensen, {An Introduction to Frames and Riesz Bases}, Springer, 2016.
\newblock \href {https://doi.org/10.2307/30037432}
  {\path{doi:10.2307/30037432}}.

\bibitem{Unser2014}
M.~Unser, P.~D. Tafti, {An Introduction to Sparse Stochastic Processes},
  Cambridge University Press, 2014.
\newblock \href {https://doi.org/10.1017/CBO9781107415805}
  {\path{doi:10.1017/CBO9781107415805}}.

\bibitem{Strang1971}
G.~Strang, G.~Fix, {A Fourier Analysis of the Finite Element Variational
  Method}, in: Constructive Aspects of Functional Analysis, Springer, 2011, pp.
  793--840.
\newblock \href {https://doi.org/10.1007/978-3-642-10984-3_7}
  {\path{doi:10.1007/978-3-642-10984-3_7}}.

\bibitem{DeBoor1998}
C.~de~Boor, R.~A. DeVore, A.~Ron, {Approximation Orders of FSI Spaces in
  L2(Rd)}, Constructive Approximation 14~(4) (1998) 631--652.
\newblock \href {https://doi.org/10.1007/s003659900094}
  {\path{doi:10.1007/s003659900094}}.

\bibitem{Unser1997}
M.~Unser, I.~Daubechies, {On the Approximation Power of Convolution-Based Least
  Squares versus Interpolation}, IEEE Transactions on Signal Processing 45~(7)
  (1997) 1697--1711.
\newblock \href {https://doi.org/10.1109/78.599940}
  {\path{doi:10.1109/78.599940}}.

\bibitem{Aziznejad2019}
S.~Aziznejad, A.~Naderi, M.~Unser, {Optimal Spline Generators for Derivative
  Sampling}, in: 2019 13th International conference on Sampling Theory and
  Applications (SampTA), IEEE, 2019, pp. 1--4.
\newblock \href {https://doi.org/10.1109/SampTA45681.2019.9030990}
  {\path{doi:10.1109/SampTA45681.2019.9030990}}.

\bibitem{Antonelli2014}
M.~Antonelli, C.~V. Beccari, G.~Casciola, {A General Framework for the
  Construction of Piecewise-Polynomial Local Interpolants of Minimum Degree},
  Advances in Computational Mathematics 40~(4) (2014) 945--976.
\newblock \href {https://doi.org/10.1007/s10444-013-9335-y}
  {\path{doi:10.1007/s10444-013-9335-y}}.

\bibitem{Ranirina2019}
D.~Ranirina, J.~de~Villiers, {On Hermite Vector Splines and Multi-Wavelets},
  Journal of Computational and Applied Mathematics 349 (2019) 366--378.
\newblock \href {https://doi.org/10.1016/j.cam.2018.08.007}
  {\path{doi:10.1016/j.cam.2018.08.007}}.

\bibitem{Lachance1990}
M.~Lachance, {An introduction to splines for use in computer graphics and
  geometric modeling}, Computer Vision, Graphics, and Image Processing (1990).
\newblock \href {https://doi.org/10.1016/0734-189x(90)90071-3}
  {\path{doi:10.1016/0734-189x(90)90071-3}}.

\bibitem{Alfeld1985}
P.~Alfeld, {On the Dimension of Multivariate Piecewise Polynomials}, Numerical
  analysis (1986) 1--23.

\bibitem{Petrinovic2008}
D.~Petrinovi{\'{c}}, {Causal Cubic Splines: Formulations, Interpolation
  Properties and Implementations}, IEEE Transactions on Signal Processing
  56~(11) (2008) 5442--5453.
\newblock \href {https://doi.org/10.1109/TSP.2008.929133}
  {\path{doi:10.1109/TSP.2008.929133}}.

\bibitem{Langtangen2019}
H.~P. Langtangen, K.-A. Mardal, {Function Approximation by Finite Elements},
  Springer International Publishing, Cham, 2019, pp. 69--129.
\newblock \href {https://doi.org/10.1007/978-3-030-23788-2_3}
  {\path{doi:10.1007/978-3-030-23788-2_3}}.

\bibitem{debarre2019}
T.~Debarre, S.~Aziznejad, M.~Unser, {Hybrid-Spline Dictionaries for
  Continuous-Domain Inverse Problems}, IEEE Transactions on Signal Processing
  67~(22) (2019) 5824--5836.
\newblock \href {https://doi.org/10.1109/TSP.2019.2944754}
  {\path{doi:10.1109/TSP.2019.2944754}}.

\end{thebibliography}

\end{document}